\documentclass{amsart}
\usepackage{amssymb}
\usepackage{amsmath}
\usepackage{amsfonts}

\newtheorem{theorem}{Theorem}
\theoremstyle{plain}

\newtheorem{corollary}{Corollary}

\newtheorem{lemma}{Lemma}

\newtheorem{proposition}{Proposition}
\newtheorem{remark}{Remark}

\numberwithin{equation}{section}
\input{tcilatex}

\begin{document}
\title[]{Remarks on the uniqueness of weak solutions of the incompressible
Navier-Stokes equations}
\author{Kamal N. Soltanov}
\address{{\small National Academy of Sciences of Azerbaijan, Baku, AZERBAIJAN%
}}
\email{sulta.kamal.n@gmail.com}
\urladdr{https://www.researchgate.net/profile/Kamal-Soltanov/research}

\begin{abstract}
This article studies the uniqueness of the weak solution of the
incompressible Navier-Stokes Equations in the 3-dimensional case. Here, the
investigation is provided using two different approaches. The first (the
main) result is obtained for given functions possessing a certain smoothness
using the new approach. The second result is without the complementary
conditions but is, in some sense, the "local" result investigated by another
approach. In addition, here the solvability and uniqueness of the weak
solutions to auxiliary problems lead out from the main problem are
investigated.
\end{abstract}

\maketitle

\section{\protect\bigskip \label{Sec_1}Introduction}

In this article, in the 3-dimensional case, the uniqueness of a weak
solution to the mixed problem with the Dirichlet homogeneous boundary
condition for the incompressible Navier-Stokes equation is investigated.

As it is well-known in the article \cite{Ler1} shown that the Navier-Stokes
equations have a weak solution with suitable properties in the 3-dimensional
case (see, \cite{Hop}, \cite{Lad1}, \cite{MajBer}, \cite{Con1}, \cite{Fef1}, 
\cite{Lio1}, etc.). In articles \cite{LioPro}, \cite{Lad2} proved the
uniqueness of a weak solution to this problem in the 2-dimension case,
without complementary conditions (see, also \cite{Lad1}, \cite{Lio1}), but
in the 3-dimensional case a result of such type doesn't exist. It should be
noted that in the 3-dimensional case, the uniqueness also was studied, but
under complementary conditions on the smoothness of the solution (see, e.g. 
\cite{Hop}, \cite{Lad2}, \cite{Lio1}, \cite{Tem1}, \cite{Ku}, \cite{OL}, 
\cite{Sch1}, etc.), and also there exist such setting of problems for this
equation and the Euler equations, under which the uniqueness isn't (see, 
\cite{Buc-Vic}, \cite{Sch1}, \cite{Shn1}).

The study of the different problems for the Navier--Stokes equation
dedicated sufficiently many works, where the various questions, including
e.g. also the different properties of solutions to these problems (see, \cite%
{Sch2}, \cite{CafKohNir}, \cite{Lin1}, \cite{Lio1}, \cite{Lad1}, \cite%
{ChL-RiMay}, \cite{Gal} \cite{Lio1}, \cite{Lad1}, \cite{Lin1}, \cite{Fef1}, 
\cite{FoiManRosTem}, \cite{FoiRosTem1}, \cite{FoiRosTem2}, \cite{FoiRosTem3}%
, \cite{GlaSveVic}, \cite{HuaWan}, \cite{PerZat}, \cite{Sol1}, \cite{Sol2}, 
\cite{Tem1}, etc.). Studied also different modifications of Navier--Stokes
equation (see, e.g. \cite{Lad1}, \cite{Lio1}, \cite{Sol3}, etc.). \ 

We should note the result of the article \cite{Fur} that is close to the
first main result of this article. In work \cite{Fur} considered the problem
for the equation without the pressure $p\left( t,x\right) $ (equation from (%
\ref{P}) that becomes the problem for studying the uniqueness of a weak
solution $u\left( t,x\right) $ (the speed), which is investigated under some
smoothness conditions. Here is proven the uniqueness of a weak solution $%
u\left( t,x\right) $ to the considered problem for any $f\left( t,x\right) $
from the everywhere dense subset, in a certain sense, of the space $%
L^{2}\left( 0,T:H^{-1/2}\left( 
\Omega
\right) \right) $ when $u_{0}\in H^{1/2}\left( \Omega \right) $ is arbitrary.

Well-known that studying the existence of the weak solution of the
considered problem uses the approach Hopf-Leray (taking into account the
result of de Rham). Therefore, from the definition of a weak solution to the
problem for the Navier-Stokes equation obtained the problem that, roughly
speaking, will have the form 
\begin{equation}
\frac{\partial u}{\partial t}-\nu \Delta u+u\cdot \nabla u=f\left(
t,x\right) ,\ \func{div}u=0,\ u\left( 0,x\right) =u_{0}\left( x\right) ,\
u\left\vert \ _{\left( 0,T\right) \times \partial \Omega }\right. =0
\label{P}
\end{equation}%
for each $p\left( t,\cdot \right) \in L_{2}\left( \Omega \right) $, due to
the condition $\func{div}u=0$. Since this problem only the gradient of
pressure $p$ contains then it can be eliminated from the Navier-Stokes
equations as is noted in many works (see, e.g. \cite{Tao}, \cite{Con},\cite%
{Tem1}, \cite{Lio1}, \cite{Bat}, \cite{Vas}, etc.). So, we study here the
uniqueness of a weak solution $u\left( t,x\right) $ of the problem (\ref{2.1}%
) - (\ref{2.3}) due to the above property of the problem. For the study of
the uniqueness of a weak solution $u$, we use the well-known formulation of
this problem. In what follows, we will base on results from the
above-mentioned works concerning the existence of a weak solution to the
posed problem.

This article is organized as follows. In section 2, the preliminary
information and the explanation of the relationship between problems for the
Navier-Stokes equations and (\ref{2.1}) - (\ref{2.3}) are reduced, moreover,
the main results, and also the necessary auxiliary results are proved. In
section 3, the auxiliary problems are determined, which is constructed using
the problem (\ref{2.1}) - (\ref{2.3}) as, in some sense, projection of this
problem to the cross-sections. In Section 4, the existence of the weak
solution and, in Section 5, the uniqueness of the weak solution of the
auxiliary problem are studied. In Section 6, the first main result (Theorem %
\ref{Th_1}) is proved. In Section 7, the complementary investigation is
provided and the "local" result on the uniqueness of the weak solution of
the problem (\ref{2.1}) - (\ref{2.3}) is proved by certain modification of
the known approach.

\section{\label{Sec_2}Preliminaries and main results}

In this section, the background material is briefly reminded, and some
notations are introduced. Moreover, the main results are reduced, and also
the necessary auxiliary results are proved.

Let $\Omega \subset R^{3}$ be an open bounded local Lipschitz domain (i.e.
from the class $Lip_{loc}$) and $Q^{T}\equiv \left( 0,T\right) \times \Omega 
$, $T>0$ be a number. As is usual in the study of the Navier-Stokes
equation, we denote by $V\left( \Omega \right) $ and $H\left( \Omega \right) 
$ the spaces that are determined as the closure in the topology of $\left(
W_{0}^{1,2}\left( \Omega \right) \right) ^{3}$ and of $\left( L^{2}\left(
\Omega \right) \right) ^{3}$, respectively, of the functions class 
\begin{equation*}
\left\{ \varphi \left\vert \ \varphi \in \left( C_{0}^{\infty }\left( \Omega
\right) \right) ^{3},\right. \func{div}\varphi =0\right\} ,
\end{equation*}%
where $W_{0}^{1,2}\left( \Omega \right) $ is the Sobolev space, $L^{2}\left(
\Omega \right) $ is the Lebesgue space.

We also denote by $H^{1/2}\left( \Omega \right) $ the vector space
determined as $H\left( \Omega \right) $ in the topology of 
\begin{equation*}
\left( W^{1/2,2}\left( \Omega \right) \right) ^{3}\equiv \left\{ v\left\vert
\ v_{i}\in W^{1/2,2}\left( \Omega \right) ,\right. i=1,2,3\right\} ,\quad
v=\left( v_{1},v_{2},v_{3}\right) ,
\end{equation*}%
where $W^{1/2,2}\left( \Omega \right) $ is the Sobolev-Slobodeskij space
(see, \cite{LioMag}, etc.). As well-known, the trace for the function of the
space $H^{1/2}\left( \Omega \right) $ is definite for each sufficiently
smooth surface from $\Omega $ (see, e.g. \cite{LioMag}, \cite{BesIlNik} and
references therein).

The dual space $V\left( \Omega \right) $ is determined as $V^{\ast }\left(
\Omega \right) $ and is the closure of the\ linear continuous functionals
defined on $V\left( \Omega \right) $ in the suitable sense. As it is
well-known, in this case, $H\left( \Omega \right) $ and $V\left( \Omega
\right) $ are the Hilbert spaces, and relations take place

\begin{equation*}
V\left( \Omega \right) \subset H\left( \Omega \right) \equiv H^{\star
}\left( \Omega \right) \subset V^{\ast }\left( \Omega \right) .
\end{equation*}%
Let $V\left( Q^{T}\right) $ be a space determined as 
\begin{equation*}
V\left( Q^{T}\right) \equiv L^{2}\left( 0,T;V\left( \Omega \right) \right)
\cap L^{\infty }\left( 0,T;H\left( \Omega \right) \right)
\end{equation*}%
and $\mathcal{V}\left( Q^{T}\right) $ be a space determined as 
\begin{equation*}
\mathcal{V}\left( Q^{T}\right) \equiv V\left( Q^{T}\right) \cap
W^{1,4/3}\left( 0,T;V^{\ast }(%
\Omega
)\right) .
\end{equation*}

So, from the above reasons follows that for the study of the posed problem
enough to study this question for the following problem

\begin{equation}
\frac{\partial u_{i}}{\partial t}-\nu \Delta u_{i}+\underset{j=1}{\overset{n}%
{\sum }}u_{j}\frac{\partial u_{i}}{\partial x_{j}}=f_{i}\left( t,x\right)
,\quad i=\overline{1,n},\ \nu >0  \label{2.1}
\end{equation}%
\begin{equation}
\func{div}u=\underset{i=1}{\overset{n}{\sum }}\frac{\partial u_{i}}{\partial
x_{i}}=\underset{i=1}{\overset{n}{\sum }}D_{i}u_{i}=0,\quad x\in \Omega
\subset R^{n},\ t>0,  \label{2.2}
\end{equation}%
\begin{equation}
u\left( 0,x\right) =u_{0}\left( x\right) ,\quad x\in \Omega ;\quad
u\left\vert \ _{\left( 0,T\right) \times \partial \Omega }\right. =0.
\label{2.3}
\end{equation}

Here our main problem is the investigation of the posed question in case $%
n=3 $. On the existence of the solution and properties of the existing
solution of this problem dedicated sufficiently many works (see, e.g. books 
\cite{Lio1}, \cite{Tem1}, \cite{Lad1}, \cite{FoiManRosTem}, \cite{Gal}, \cite%
{MajBer}, \cite{Con}, \cite{Bat}, \cite{Vas}, etc. where the properties of
this problem were explained enough clearly).\ 

As it is well-known, problem (\ref{2.1}) - (\ref{2.3}) possesses a weak
solution $u$ in the space $\mathcal{V}\left( Q^{T}\right) \times L^{2}\left(
Q^{T}\right) $ for each $u_{0}\left( x\right) ,$ $f(x,t)$, which are
contained in the spaces $u_{0}\in H\left( \Omega \right) ,$ $f\in
L^{2}\left( 0,T;V^{\ast }\left( \Omega \right) \right) $ (see, e.g. \cite%
{Lio1}, \cite{Tem1}, \cite{Gal}, \cite{Con}, \cite{Bat}, \cite{Vas}, etc.).
Moreover, the solution $u$ is a weakly continuous function concerning
variable $t$.

As usual, a function $u\in \mathcal{V}\left( Q^{T}\right) $ is called a weak
solution of problem (\ref{2.1}) - (\ref{2.3}) if $u\left( t\right) $ almost
everywhere in $\left( 0,T\right) $ satisfies the following equation 
\begin{equation}
\frac{d}{dt}\left\langle u,v\right\rangle -\left\langle \nu \Delta
u,v\right\rangle +\left\langle \underset{j=1}{\overset{n}{\sum }}%
u_{j}D_{j}u,v\right\rangle =\left\langle f,v\right\rangle ,  \label{2.3a}
\end{equation}%
for any $v\in V\left( \Omega \right) $ and initial condition $\left\langle
u\left( 0\right) ,v\right\rangle =\left\langle u_{0},v\right\rangle $ for
any $v\in H\left( \Omega \right) $, in addition, $u\left( t\right) $\ is
weakly continuous from $\left[ 0,T\right] $ into $H\left( \Omega \right) $
(i.e. $\forall v\in H\left( \Omega \right) $, $t\longrightarrow \left\langle
u\left( t\right) ,v\right\rangle $ is a continuous scalar function, and
consequently $\left\langle u\left( 0\right) ,v\right\rangle =\left\langle
u_{0},v\right\rangle $).

\textbf{Now we can provide the main results of this article. }

\begin{theorem}
\label{Th_1}Let $\Omega \subset R^{3}$ be an open bounded domain from the
class $Lip_{loc}$, $T>0$ be a number. If given functions $u_{0}$, $f$
satisfy of conditions $u_{0}\in H^{1/2}\left( \Omega \right) $, $f\in
L^{2}\left( 0,T;H^{1/2}\left( \Omega \right) \right) $ then weak solution $%
u\left( t,x\right) $ of the problem (\ref{2.1}) - (\ref{2.3}) is unique, for 
$t\in \left( 0,T\right] $.
\end{theorem}

We should note that the well-known theorems on weak solvability assumed the
following conditions for the given functions $u_{0}\in H\left( \Omega
\right) ,$ $f\in L^{2}\left( 0,T;V^{\ast }\left( \Omega \right) \right) $,
but unlike theirs, the Theorem \ref{Th_1} assumes that these functions
should be a few smooth, namely, these be satisfied conditions $u_{0}\in
H^{1/2}\left( \Omega \right) $, $f\in L^{2}\left( 0,T;H^{1/2}\left( \Omega
\right) \right) $.

The main result of case (2) in this paper is the following uniqueness
theorem.

\begin{theorem}
\label{Th_2}Let $\Omega \in R$ be an open bounded domain from the class $%
Lip_{loc}$, $T>0$ be a number, and $\left( u_{0},f\right) \in H\left( \Omega
\right) \times L^{2}\left( 0,T;V^{\ast }\left( \Omega \right) \right) $. Let
the problem (\ref{2.1}) - (\ref{2.3}) has a weak solution $u\left(
t,x\right) $ that belongs to the space $u\in \mathcal{V}\left( Q^{T}\right) $%
. Then a weak solution u(t,x) is unique if either $\underset{\Omega }{\int }%
\left\vert B\left( w,w\right) \right\vert dx\geq 0$ or $\nu \lambda _{1}^{%
\frac{1}{4}}\geq c^{2}\underset{i=1}{\overset{3}{\sum }}\left\Vert
D_{i}u_{j}\left( t\right) \right\Vert _{2}$ be fulfilled for $t\in \left( 0,T%
\right] $, where $\lambda _{1}$ is minimum of the spectrum (or the first
eigenvalue) of the operator Laplace.
\end{theorem}

We should note that since $w\in R^{3}$ we used the relation between the
obtained quadratic form and the surface in $R^{3}$.

We will go over to the main question: Let problem (\ref{2.1}) - (\ref{2.3})
have two different weak solutions $u,v\in \mathcal{V}\left( Q^{T}\right) $,
then we get that function $w(t,x)=u(t,x)-v(t,x)$ is the weak solution to the
following problem 
\begin{equation}
\frac{1}{2}\frac{\partial }{\partial t}\left\Vert w\right\Vert _{2}^{2}+\nu
\left\Vert \nabla w\right\Vert _{2}^{2}+\underset{j,k=1}{\overset{3}{\sum }}%
\left\langle \frac{\partial v_{k}}{\partial x_{j}}w_{k},w_{j}\right\rangle
=0,  \label{2.4a}
\end{equation}%
\begin{equation}
w\left( 0,x\right) \equiv w_{0}\left( x\right) =0,\quad x\in \Omega ;\quad
w\left\vert \ _{\left( 0,T\right) \times \partial \Omega }\right. =0,
\label{2.4}
\end{equation}%
where $\left\langle g,h\right\rangle =\underset{i=1}{\overset{3}{\sum }}%
\underset{\Omega }{\int }g_{i}h_{i}dx$ for any $g,h\in \left( H\left( \Omega
\right) \right) ^{3}$, or $g\in V\left( \Omega \right) $ and $h\in V^{\ast
}\left( \Omega \right) $, respectively. So, for the proof of the uniqueness
of a weak solution, it needs to show that $w\equiv 0$ in the sense of the
appropriate space. In other hand, if $u,v\in \mathcal{V}\left( Q^{T}\right) $
two different weak solutions to problem (\ref{2.1}) - (\ref{2.3}) then $w\in 
\mathcal{V}\left( Q^{T}\right) $ and $w(t,x)\neq 0$, at least, on some
subdomain $Q_{1}^{T}$ of $Q^{T}$, which Lebesgue measure great than zero. In
other words, there exists such subdomain $\Omega _{1}$ and interval $\left(
t_{1},t_{2}\right) \subseteq \left( 0,T\right] $ that 
\begin{equation*}
Q_{1}^{T}\subseteq \left( t_{1},t_{2}\right) \times \Omega _{1}\subseteq
Q^{T},\quad mes_{4}\left( Q_{1}^{T}\right) >0,
\end{equation*}%
i.e. 
\begin{equation}
mes_{4}\left( \left\{ (t,x)\in Q^{T}\left\vert \ \left\vert
w(t,x)\right\vert \right. >0\right\} \right) =mes_{4}\left( Q_{1}^{T}\right)
>0,  \label{2.5}
\end{equation}%
where $mes_{4}\left( Q_{1}^{T}\right) $ denotes the measure of $Q_{1}^{T}$
in $R^{4}$ (i.e. $mes_{k}$ denotes the Lebesgue measure on $k-$dimensional
subspace $R^{n}$). Consequently, $mes_{3}(\Omega _{1})>0$.

So, for proof of the uniqueness, we will employ a different procedure,
namely, use the above-mentioned argument. Therefore, we begin to prove some
auxiliary results.

\begin{lemma}
\label{L_2.1}Let $G\subset R^{n}$ be a subset that is Lebesgue measurable,
then the following statements are equivalent:

1) $\infty >mes_{n}\left( G\right) >0;$

2) there exist such subsets $I\subset R^{1}$, $\infty >mes_{1}\left(
I\right) >0$ and $G_{\beta }\subset L_{\beta ,n-1}$, $\infty
>mes_{n-1}\left( G_{\beta }\right) >0$ that $G=\underset{\beta \in I}{\cup }%
G_{\beta }\cup N$, where $N$ is the set with $mes_{n-1}\left( N\right) =0$,
and $L_{\beta ,n-1}$ is the hyperplane of $R^{n}$, with $co\dim _{n}L_{\beta
,n-1}=1$, and $G\cap L_{\gamma ,n-1}\neq \varnothing $ for any $\beta \in I$
, which is generated by some fixed vector $y_{0}\in R^{n}$ and defined as
follow 
\begin{equation*}
L_{\beta ,n-1}\equiv \left\{ y\in R^{n}\left\vert \ \left\langle
y_{0},y\right\rangle =\beta \right. \right\} ,\quad \forall \beta \in I.
\end{equation*}
\end{lemma}

\begin{proof}
Let $mes_{n}\left( G\right) >0$ and consider the class of hyperplanes $%
L_{\gamma ,n-1}$ for which $G\cap L_{\gamma ,n-1}\neq \varnothing $ and $%
\gamma \in I_{1}$, where $I_{1}\subset R^{1}$\ be some subset. It is clear
that 
\begin{equation*}
G\equiv \underset{\gamma \in I_{1}}{\bigcup }\left\{ x\in G\cap L_{\gamma
,n-1}\left\vert \ \gamma \in I_{1}\right. \right\} .
\end{equation*}%
Then there exists a subclass of hyperplanes $\left\{ L_{\gamma
,n-1}\left\vert \ \gamma \in I_{1}\right. \right\} $, for which

$mes_{n-1}\left( G\cap L_{\gamma ,n-1}\right) >0$ is fulfilled. The number
of such type hyperplanes cannot be less than countable or equal it because $%
mes_{n}\left( G\right) >0$, moreover, this subclass of $I_{1}$\ must possess
the $R^{1}$ measure greater than $0$ since $mes_{n}\left( G\right) >0$.
Indeed, let $I_{1,0}$ be this subclass and $mes_{1}\left( I_{1,0}\right) =0$%
. In this case, we get the subset 
\begin{equation*}
\left\{ \left( \gamma ,y\right) \in I_{1,0}\times G\cap L_{\gamma
,n-1}\left\vert \ \gamma \in I_{1,0},y\in G\cap L_{\gamma ,n-1}\right.
\right\} \subset R^{n}
\end{equation*}%
where $mes_{n-1}\left( G\cap L_{\gamma ,n-1}\right) >0$ for all $\gamma \in
I_{1,0}$, but $mes_{1}\left( I_{1,0}\right) =0$, then 
\begin{equation*}
mes_{n}\left( \left\{ \left( \gamma ,y\right) \in I_{1,0}\times G\cap
L_{\gamma ,n-1}\left\vert \ \gamma \in I_{1,0}\right. \right\} \right) =0.
\end{equation*}%
On the other hand if $mes_{n-1}\left( G\cap L_{\gamma ,n-1}\right) =0$ for
all $\gamma \in I_{1}-I_{1,0}$ then 
\begin{equation*}
mes_{n}\left( \left\{ \left( \gamma ,y\right) \in I_{1}\times G\cap
L_{\gamma ,n-1}\left\vert \ \gamma \in I_{1}\right. \right\} \right) =0,
\end{equation*}%
whence follows 
\begin{equation*}
mes_{n}\left( G\right) =mes_{n}\left( \left\{ \left( \gamma ,y\right) \in
I_{1}\times G\cap L_{\gamma ,n-1}\left\vert \ \gamma \in I_{1}\right.
\right\} \right) =0.
\end{equation*}

But this contradicts the condition $mes_{n}\left( G\right) >0$.
Consequently, statement 2 holds.

Let statement 2 hold. It is clear that the class of hyperplanes $L_{\beta
,n-1}$ defined in such a way are parallel and also we can define the class
of subsets of $G$ as its cross-section with hyperplanes, i.e. in the form: $%
G_{\beta }\equiv G\cap L_{\beta ,n-1}$, \ \ $\beta \in I$. Then $G_{\beta
}\neq \varnothing $ and we can write $G_{\beta }\equiv G\cap L_{\beta ,n-1}$%
, $\beta \in I$, moreover $G\equiv \underset{\beta \in I}{\bigcup }\left\{
x\in G\cap L_{\beta ,n-1}\left\vert \ \beta \in I\right. \right\} \cup N$.
Whence we get 
\begin{equation*}
G\equiv \left\{ \left( \beta ,x\right) \in I\times G\cap L_{\beta
,n-1}\left\vert \ \beta \in I,x\in G\cap L_{\beta ,n-1}\right. \right\} \cup
N.
\end{equation*}

Consequently, $mes_{n}\left( G\right) >0$ by virtue of conditions: $%
mes_{1}\left( I\right) >0$ and $mes_{n-1}\left( G_{\beta }\right) >0$ for
any $\beta \in I$.
\end{proof}

Lemma \ref{L_2.1} shows that for the study of the measure of some subset $%
\Omega
\subseteq R^{n}$ it is enough to study its stratifications by a class of
corresponding hyperplanes.

\begin{lemma}
\label{L_2.2}Let problem (\ref{2.1}) -(\ref{2.3}) have, at least, two
different solutions $u,v$ that are contained in $\mathcal{V}\left(
Q^{T}\right) $ and assume that $Q_{1}^{T}\subseteq Q^{T}$ is one of a
subdomain of $Q^{T}$ where $u$ and $v$ are different. Then there exists, at
least, one class of parallel hyperplanes $L_{\alpha }$, $\alpha \in
I\subseteq \left( \alpha _{1},\alpha _{2}\right) \subset R^{1}$ ($\alpha
_{2}>\alpha _{1}$)\ with $co\dim _{R^{3}}L_{\alpha }=1$ such, that $u\neq v$
on $Q_{L_{\alpha }}^{T}\equiv \left[ \left( 0,T\right) \times \left( \Omega
\cap L_{\alpha }\right) \right] \cap Q_{1}^{T}$, and vice versa, where $%
mes_{1}\left( I\right) >0$, $mes_{2}\left( \Omega \cap L_{\alpha }\right) >0$
and $L_{\alpha }$ are hyperplanes which are defined as follows: there is
such vector $x_{0}\in S_{1}^{R^{3}}\left( 0\right) $ that 
\begin{equation*}
L_{\alpha }\equiv \left\{ x\in R^{3}\left\vert \ \left\langle
x_{0},x\right\rangle =\alpha ,\right. \ \forall \alpha \in I\right\} .
\end{equation*}
\end{lemma}

\begin{proof}
Let problem (\ref{2.1}) - (\ref{2.3}) have two different solutions $u,v\in 
\mathcal{V}\left( Q^{T}\right) $ then there exist a subdomain of $Q^{T}$ on
which these solutions are different. Then there are $t_{1},t_{2}>0$ such
that 
\begin{equation}
mes_{3}\left( \left\{ x\in \Omega \left\vert \ \left\vert u\left( t,x\right)
-v\left( t,x\right) \right\vert >0\right. \right\} \right) >0  \label{2.6}
\end{equation}
holds for any $t\in J\subseteq \left[ t_{1},t_{2}\right] \subseteq \left[
0,T\right) $, where $mes_{1}\left( J\right) >0$ by the virtue of the
condition 
\begin{equation}
mes_{4}\left( \left\{ (t,x)\in Q^{T}\left\vert \ \left\vert
u(t,x)-v(t,x)\right\vert \right. >0\right\} \right) >0  \label{2.7}
\end{equation}%
and of Lemma \ref{L_2.1}.

Whence follows, that there exists, at least, one class of such parallel
hyperplanes $L_{\alpha }$, $\alpha \in I\subseteq \left( \alpha _{1},\alpha
_{2}\right) \subset R^{1}$ that $co\dim _{R^{3}}L_{\alpha }=1$ and 
\begin{equation*}
mes_{2}\left( \left\{ x\in \Omega \cap L_{\alpha }\left\vert \ \left\vert
u\left( t,x\right) -v\left( t,x\right) \right\vert >0\right. \right\}
\right) >0,\ \forall \alpha \in I
\end{equation*}%
hold for $\forall t\in J$, where subsets $I$ and $J$ are satisfied
inequations: $mes_{1}\left( I\right) >0$, $mes_{1}\left( J\right) >0$, and
also (\ref{2.7}) holds, by virtue of (\ref{2.6}). This proves the "if" part
of the Lemma.

Now consider the converse assertion. Let there exist a class of hyperplanes $%
L_{\alpha }$, $\alpha \in I_{1}\subseteq \left( \alpha _{1},\alpha
_{2}\right) \subset R^{1}$ with $co\dim _{R^{3}}L_{\alpha }=1$ that fulfills
the condition of the Lemma and the subset $I_{1}$\ satisfies of the same
condition as $I$. Then there exist, at least, one subset $J_{1}$ of $\left[
0,T\right) $ such that $mes_{1}\left( J_{1}\right) >0$ and the inequation $%
u\left( t,x\right) \neq v\left( t,x\right) $ holds onto $Q_{2}^{T}$ with $%
mes_{4}\left( Q_{2}^{T}\right) >0$, which is defined as $Q_{2}^{T}\equiv
J_{1}\times U_{L}$, where 
\begin{equation}
U_{L}\equiv \underset{\alpha \in I_{1}}{\bigcup }\left\{ x\in \Omega \cap
L_{\alpha }\left\vert \ u\left( t,x\right) \neq v\left( t,x\right) \right.
\right\} \subset \Omega ,\ t\in J_{1}  \label{2.8}
\end{equation}%
for which the inequation $mes_{R^{3}}\left( U_{L}\right) >0$ is fulfilled by
the condition and of Lemma \ref{L_2.1}.

So, we get 
\begin{equation*}
u\left( t,x\right) \neq v\left( t,x\right) \text{ onto }Q_{2}^{T}\equiv
J_{1}\times U_{L},\text{ with }mes_{4}\left( Q_{2}^{T}\right) >0.
\end{equation*}%
Consequently, we obtain the fact that $u\left( t,x\right) $ and $v\left(
t,x\right) $ are different functions in $\mathcal{V}\left( Q^{T}\right) $.
\end{proof}

It is not difficult to see that the result of Lemma \ref{L_2.2} is
independent of the assumption: $Q_{1}^{T}\subset Q^{T}$ or $Q_{1}^{T}=Q^{T}$.

\section{\label{Sec_3} Definition and property of the auxiliary problem}

To study the posed question in this section, we will transform the problem (%
\ref{2.1}) - (\ref{2.3}) posed on $Q^{T}$ to the auxiliary problems that are
formulated on the cross-sections of the domain $Q^{T}\equiv \left(
0,T\right) \times \Omega $. In this case, we assume the following
complementary condition on the given functions that $u_{0}\in H^{\frac{1}{2}%
}\left( \Omega \right) $, $f\in L^{2}\left( 0,T;H^{\frac{1}{2}}\left( \Omega
\right) \right) $.

Since $\Omega \subset R^{3}$ is the open bounded domain from the class $%
Lip_{loc}$ then each point $x_{j}\in \partial \Omega $, has such open
neighborhood $U_{j}$ that $U_{j}^{\prime }=\overline{\Omega }\cap U_{j}$
that is from the class $Lip$. Consequently,\textit{\ }for every
"cross-section", $\Omega _{L}\equiv \Omega \cap L\neq \varnothing $\ of $%
\Omega
$\ with arbitrary hyperplane $L$\ exists, at least, one coordinate subspace (%
$\left( x_{j},x_{k}\right) $),\ which possesses a domain that we denote as $%
P_{x_{i}}\Omega _{L}$\ (or union of domains) from the class $Lip_{loc}$,
since $\partial \Omega _{L}\equiv \partial \Omega \cap L\neq \varnothing $
and defined\ with the affine representation by $\Omega _{L}$,\textit{\ }in
addition, isomorphic to $\partial P_{x_{i}}\Omega _{L}$, i.e. $\partial
\Omega _{L}\Longleftrightarrow \partial P_{x_{i}}\Omega _{L}$.

Thus, using of the representation $P_{x_{i}}L$ of the hyperplane $L$ we get
that $\Omega _{L}$ can be written in the form $P_{x_{i}}\Omega _{L}$,
therefore, an integral on $\Omega _{L}$ also will be defined by the
respective representation, i.e. as the integral on $P_{x_{i}}\Omega _{L}$.

It should be noted that $\Omega _{L}$ can consist of many parts, then $%
P_{x_{i}}\Omega _{L}$ will be such as $\Omega _{L}$. In this case, $\Omega
_{L}$ will be as the union of domains, then the following relation will be
held 
\begin{equation*}
\Omega _{L}=\underset{r=1}{\overset{m}{\cup }}\Omega _{L}^{r}\
\Longleftarrow \Longrightarrow \ P_{x_{i}}\Omega _{L}=\underset{r=1}{\overset%
{m}{\cup }}P_{x_{i}}\Omega _{L}^{r},\quad \infty >m\geq 1.
\end{equation*}%
Therefore, each of $P_{x_{i}}\Omega _{L}^{r}$ will be the domain and these
can be investigated separately, as $\Omega _{L}^{r}\subset \Omega $ and $%
\partial \Omega _{L}^{r}\subset \partial \Omega $.

So, we will define subdomains of $Q^{T}\equiv \left( 0,T\right) \times
\Omega $ as follows $Q_{L}^{T}\equiv \left( 0,T\right) \times \left( \Omega
\cap L\right) $, where $L$ is an arbitrarily fixed hyperplane of the
dimension two and $\Omega \cap L\neq \varnothing $.

Consequently, we will investigate the uniqueness of the problem (\ref{2.1})
- (\ref{2.3}) on the "cross-section" $Q_{L}^{T}$ defined according to the
"cross-section" of $\Omega $. Denote by $\Omega _{L}$ of the "cross-section" 
$\Omega _{L}\equiv \Omega \cap L\neq \varnothing $, $mes_{R^{2}}\left(
\Omega _{L}\right) >0$ (e.g. $L$ can be $L\equiv \left\{ \left(
x_{1},x_{2},0\right) \left\vert \ x_{1},x_{2}\in R^{1}\right. \right\} $).
In other words, we can determine $L$ as 
\begin{equation*}
L\equiv \left\{ x\in R^{3}\left\vert \ \left\langle a,x\right\rangle
=a_{1}x_{1}+a_{2}x_{2}+a_{3}x_{3}=b\right. \right\} ,
\end{equation*}%
where $a\in S_{1}^{R^{3}}\left( 0\right) $ and $b\in R^{1}$ are arbitrary
fixed, furthermore, each $a\in S_{1}^{R^{3}}\left( 0\right) $ and $b\in
R^{1} $ define of single $L_{b}\left( a\right) $ and vice versa. Whence
follows that $a_{3}x_{3}=b-a_{1}x_{1}-a_{2}x_{2}$, if assume $a_{3}\neq 0$
then $x_{3}=\frac{1}{a_{3}}\left( b-a_{1}x_{1}-a_{2}x_{2}\right) $, by
renaming coefficients we obtain $x_{3}\equiv \psi _{3}\left(
x_{1},x_{2}\right) =b-a_{1}x_{1}-a_{2}x_{2}$.

Since we will investigate the problem (\ref{2.1}) - (\ref{2.3}) on $%
Q_{L}^{T} $, in the beginning, we need to define the problem deriving under
changing of the domain, on which was formulated the problem. Clearly, under
this projection, some of the expressions in the problem (\ref{2.1}) - (\ref%
{2.3}) will be changed according to the above relation between the
independent variables. Let $L$ be an arbitrary hyperplane intersecting with $%
\Omega $, i.e. $\Omega _{L}\neq \varnothing $ and $u\in \mathcal{V}\left(
Q^{T}\right) $ is the weak solution of the problem (\ref{2.1}) - (\ref{2.3}%
). We will assume functions $u_{0}$ and $f$ satisfy conditions of the
Theorem \ref{Th_1}, namely, $u_{0}\in H^{1/2}\left( \Omega \right) $ and $%
f\in L^{2}\left( 0,T;H^{1/2}\left( \Omega \right) \right) $ so that these
functions to be correctly defined on $\left( 0,T\right] \times \Omega _{L}$.

Now we will provide some reasoning and calculations to formulate the problem
generated by the "projection" of the problem (\ref{2.1}) - (\ref{2.3}) onto $%
\left( 0,T\right] \times \Omega _{L}$. Thus, we will derive expressions 
\begin{equation}
2D_{3}\equiv \frac{\partial x_{1}}{\partial x_{3}}D_{1}+\frac{\partial x_{2}%
}{\partial x_{3}}D_{2}=-\frac{1}{a_{1}}D_{1}-\frac{1}{a_{2}}D_{2}\quad \&
\label{3.1}
\end{equation}%
\begin{equation}
4D_{3}^{2}=\frac{1}{a_{1}^{2}}D_{1}^{2}+\frac{1}{a_{2}^{2}}D_{2}^{2}+\frac{2%
}{a_{1}a_{2}}D_{1}D_{2},\quad D_{i}=\frac{\partial }{\partial x_{i}},i=1,2,3,
\label{3.2}
\end{equation}%
according to the above-mentioned reasoning.

As the function $u$ belong to $\mathcal{V}\left( Q^{T}\right) $, then it $u$
is well-defined on $\left( 0,T\right] \times \Omega _{L}$, and we obtain the
following equation on $\left( 0,T\right] \times \Omega _{L}$ 
\begin{equation*}
\frac{\partial u}{\partial t}-\nu \Delta u+\underset{j=1}{\overset{3}{\sum }}%
u_{j}D_{j}u\Longrightarrow \frac{\partial u_{L}}{\partial t}-\nu \left(
D_{1}^{2}+D_{2}^{2}+D_{3}^{2}\right) u_{L}+
\end{equation*}%
\begin{equation*}
u_{L1}D_{1}u_{L}+u_{L2}D_{2}u_{L}+u_{L3}D_{3}u_{L}=\frac{\partial u_{L}}{%
\partial t}-\nu \lbrack D_{1}^{2}+D_{2}^{2}+\frac{1}{4}%
(a_{1}^{-1}D_{1}+a_{2}^{-1}D_{2})^{2}]u_{L}
\end{equation*}%
\begin{equation*}
+u_{L1}D_{1}u_{L}+u_{L2}D_{2}u_{L}-\frac{1}{2}%
u_{L3}(a_{1}^{-1}D_{1}+a_{2}^{-1}D_{2})u_{L}=f_{L}
\end{equation*}%
Hence, by introducing the notations 
\begin{equation*}
A_{1}\equiv \lbrack D_{1}^{2}+D_{2}^{2}+\frac{1}{4}%
(a_{1}^{-1}D_{1}+a_{2}^{-1}D_{2})^{2}]
\end{equation*}%
and 
\begin{equation*}
B_{1}\left( u_{L},w_{L}\right) \equiv \left[ u_{L1}D_{1}+u_{L2}D_{2}-\frac{1%
}{2}u_{L3}\left( a_{1}^{-1}D_{1}+a_{2}^{-1}D_{2}\right) \right] w_{L}
\end{equation*}%
we get the equation 
\begin{equation}
\frac{\partial u_{L}}{\partial t}-\nu A_{1}u_{L}+B_{1}\left(
u_{L},u_{L}\right) =f_{L}  \label{3.3}
\end{equation}%
on $\left( 0,T\right) \times \Omega _{L}$. Similarly is obtained 
\begin{equation}
\func{div}u_{L}=D_{1}u_{L1}+D_{2}u_{L2}-\frac{1}{2}\left(
a_{1}^{-1}D_{1}+a_{2}^{-1}D_{2}\right) u_{L3}=0,\quad x\in \Omega _{L},\ t>0
\label{3.4}
\end{equation}%
\begin{equation}
u_{L}\left( 0,x\right) =u_{L0}\left( x\right) ,\quad \left( t,x\right) \in 
\left[ 0,T\right] \times \Omega _{L};\quad u_{L}\left\vert \ _{\left(
0,T\right) \times \partial \Omega _{L}}\right. =0.  \label{3.5}
\end{equation}%
according to the representations (\ref{3.1}) and (\ref{3.2}).

Thus, we derived the problem (\ref{3.3}) - (\ref{3.5}), which will give us
the possibility to define properties of the weak solution $u$ to the problem
(\ref{2.1}) - (\ref{2.3}) on each "cross-section" $\left[ 0,T\right) \times
\Omega _{L}\equiv Q_{L}^{T}$. \ 

So, we will investigate the problem (\ref{3.3}) - (\ref{3.5}). It needs
noting that for each hyperplane $L\subset R^{3}$, there exists one, at
least, such $2-$dimensional subspace in the coordinate system of which the $%
L $ can determine as $\left( x_{i},x_{j}\right) $ and $P_{x_{k}}L=R^{2}$
(e.g. $i,j,k=1,2,3$), i.e. 
\begin{equation*}
L\equiv \left\{ x\in R^{3}\left\vert \ x=\left( x_{i},x_{j},\psi _{L}\left(
x_{i},x_{j}\right) \right) ,\right. \left( x_{i},x_{j}\right) \in
R^{2}\right\}
\end{equation*}%
and 
\begin{equation*}
\Omega \cap L\equiv \left\{ x\in \Omega \left\vert \ x=\left(
x_{i},x_{j},\psi _{L}\left( x_{i},x_{j}\right) \right) ,\right. \left(
x_{i},x_{j}\right) \in P_{x_{k}}\left( \Omega \cap L\right) \right\}
\end{equation*}%
hold, where $\psi _{L}$ is the affine function (is bijection).

Thereby, in this case functions $u(t,x),\ f(t,x)$ and$\ u_{0}(x)$ can be
represented as 
\begin{equation*}
u(t,x_{i},x_{j},\psi _{L}(x_{i},x_{j}))\equiv v(t,x_{i},x_{j})\text{, }%
f(t,x_{i},x_{j},\psi _{L}(x_{i},x_{j})\equiv \phi (x_{i},x_{j})
\end{equation*}%
and%
\begin{equation*}
u_{0}(x_{i},x_{j},\psi _{L}(x_{i},x_{j}))\equiv v_{0}(x_{i},x_{j})\text{ \ \
on }(0,T)\times P_{x_{k}}\Omega _{L},
\end{equation*}%
respectively.

So, in the problem (\ref{3.3}) - (\ref{3.5}), each of the functions can be
represented as functions from the independent variables: $t$, $x_{i}$, and $%
x_{j}$.

It is known that the Dirichlet to Neumann map is single-value mapping if the
homogeneous Dirichlet problem for the elliptic part of the equation has only
a trivial solution, i.e. zero not is an eigenvalue of this problem (see,
e.g. \cite{Nac}, \cite{BehEl}, \cite{DePrZac}, \cite{H-DR}, \cite{BelCho},
etc.). Therefore, we will show that for the auxiliary problems, this
condition is satisfied.

\begin{proposition}
\label{Pr_4.1}The homogeneous Dirichlet problem for the elliptic part of the
problem (\ref{3.3}) - (\ref{3.5}) has only a trivial solution.
\end{proposition}

\begin{proof}
If to consider the elliptic part of the problem (3.3) - (3.5) then we get\
the problem 
\begin{equation*}
-\nu A_{1}u_{L}+B_{1}\left( u_{L},u_{L}\right) \equiv -\nu \left[ \left(
D_{1}^{2}+D_{2}^{2}\right) +\frac{1}{4}\left(
a_{1}^{-1}D_{1}+a_{2}^{-1}D_{2}\right) ^{2}\right] u_{L}+\underset{j=1}{%
\overset{2}{\sum }}u_{j}D_{j}u_{L}
\end{equation*}%
\begin{equation*}
-\frac{1}{2}\left(
a_{1}^{-1}u_{L3}D_{1}u_{L}+a_{2}^{-1}u_{L3}D_{2}u_{L}\right) =0,\ x\in
\Omega _{L},\quad u_{L}\left\vert _{\ \partial \Omega _{L}}\right. =0,
\end{equation*}%
where $\Omega _{L}=\Omega \cap L$.

Let's show that this problem cannot have nontrivial solutions. This will be
proved using the method of contradiction. Let $u_{L}\in V\left( \Omega
_{L}\right) $ be the nontrivial solution of this problem then we get the
following inequation 
\begin{equation*}
0=\left\langle -\nu A_{1}u_{L}+B_{1}\left( u_{L},u_{L}\right)
,u_{L}\right\rangle _{P_{x_{3}}\Omega _{L}}=
\end{equation*}%
hence 
\begin{equation*}
=-\underset{i=1}{\overset{3}{\nu \sum }}\left\langle \left[
D_{1}^{2}+D_{2}^{2}+\frac{1}{4}\left( a_{1}^{-1}D_{1}+a_{2}^{-1}D_{2}\right)
^{2}\right] u_{Li},u_{Li}\right\rangle _{P_{x_{3}}\Omega _{L}}+
\end{equation*}%
\begin{equation*}
+\underset{i=1}{\overset{3}{\sum }}\underset{P_{x_{3}}\Omega _{L}}{\int }%
\left[ u_{L1}D_{1}+u_{L2}D_{2}-\frac{1}{2}u_{L3}\left(
a_{1}^{-1}D_{1}+a_{2}^{-1}D_{2}\right) \right] u_{Li},u_{Li}dx_{1}dx_{2}=
\end{equation*}%
\begin{equation*}
\underset{i=1}{=\overset{3}{\nu \sum }}\underset{P_{x_{3}}\Omega _{L}}{\int }%
\left[ \left\vert D_{1}u_{Li}\right\vert ^{2}+\left\vert
D_{2}u_{Li}\right\vert ^{2}\right] dx_{1}dx_{2}+\frac{1}{4}\left[ \left(
a_{1}^{-1}D_{1}+a_{2}^{-1}D_{2}\right) u_{Li}\right] ^{2}+
\end{equation*}%
\begin{equation*}
+\frac{1}{2}\underset{i=1}{\overset{3}{\sum }}\underset{P_{x_{3}}\Omega _{L}}%
{\int }\left[ D_{1}u_{L1}+D_{2}u_{L2}-\frac{1}{2}\left(
a_{1}^{-1}D_{1}+a_{2}^{-1}D_{2}\right) u_{L3}\right] \left\vert
u_{Li}\right\vert ^{2}dx_{1}dx_{2}=
\end{equation*}%
by (\ref{3.4}) 
\begin{equation*}
\underset{i=1}{=\overset{3}{\nu \sum }}\underset{P_{x_{3}}\Omega _{L}}{\int }%
\left\{ \left\vert D_{1}u_{Li}\right\vert ^{2}+\left\vert
D_{2}u_{Li}\right\vert ^{2}+\frac{1}{4}\left[ \left(
a_{1}^{-1}D_{1}+a_{2}^{-1}D_{2}\right) u_{Li}\right] ^{2}\right\}
dx_{1}dx_{2}+
\end{equation*}%
\begin{equation*}
+\frac{1}{2}\underset{i=1}{\overset{3}{\sum }}\underset{P_{x_{3}}\Omega _{L}}%
{\int }\left\vert u_{Li}\right\vert ^{2}\func{div}u_{L}dx_{1}dx_{2}=
\end{equation*}%
\begin{equation}
\underset{i=1}{\overset{3}{\nu \sum }}\underset{P_{x_{3}}\Omega _{L}}{\int }%
\left\{ \left\vert D_{1}u_{Li}\right\vert ^{2}+\left\vert
D_{2}u_{Li}\right\vert ^{2}+\frac{1}{4}\left[ \left(
a_{1}^{-1}D_{1}+a_{2}^{-1}D_{2}\right) u_{Li}\right] ^{2}\right\}
dx_{1}dx_{2}>0.  \label{3.6a}
\end{equation}%
Thus, the obtained contradiction shows that function $u_{L}$ needs to be
zero, i.e. $u_{L}=0$ holds.
\end{proof}

Consequently, the considered problem satisfies sufficient conditions, so
that to be the Dirichlet to Neumann map. It is well-known that operator $%
-\Delta :H_{0}^{1}\left( \Omega _{L}\right) \longrightarrow $ $H^{-1}\left(
\Omega _{L}\right) $ generates of the $C_{0}-$ semigroup on $H\left( \Omega
_{L}\right) $, and since inclusion $H_{0}^{1}\left( \Omega _{L}\right)
\subset H^{-1}\left( \Omega _{L}\right) $ is compact, therefore, $\left(
-\Delta \right) ^{-1}$ is the compact operator in $H^{-1}\left( \Omega
_{L}\right) $.

\section{\label{Sec_.4}Existence of Weak Solution of Problem (\protect\ref%
{3.3}) - (\protect\ref{3.5})}

So, assume the domain $\Omega \subset R^{3}$ is such as above, functions $%
u_{0}$ and $f$ satisfy the conditions 
\begin{equation*}
u_{0}\in H^{1/2}\left( \Omega \right) ,\quad f\in L^{2}\left(
0,T;H^{1/2}\left( \Omega \right) \right) ,
\end{equation*}%
then these are defined on $\Omega _{L}$, $Q_{L}^{T}$ correctly, and belong
to $H\left( \Omega _{L}\right) $, $L^{2}\left( 0,T;H\left( \Omega
_{L}\right) \right) $, respectively. The spaces $V\left( \Omega _{L}\right) $
and $H\left( \Omega _{L}\right) $ are determined using the above reasoning.
Consequently, one can determine of space $V\left( Q_{L}^{T}\right) $ as
follow%
\begin{equation*}
V\left( Q_{L}^{T}\right) \equiv L^{2}\left( 0,T;V\left( \Omega _{L}\right)
\right) \cap L^{\infty }\left( 0,T;H\left( \Omega _{L}\right) \right) .
\end{equation*}

Thus we will study the problem (\ref{3.3}) - (\ref{3.5}), where $f_{L}\in
L^{2}\left( 0,T;V^{\ast }\left( \Omega _{L}\right) \right) $ and $u_{0L}\in
H\left( \Omega _{L}\right) $. Therefore, a weak solution to the problem (\ref%
{3.3}) - (\ref{3.5}) will be understood as the following: A function $%
u_{L}\in \mathcal{V}\left( Q_{L}^{T}\right) $ is called a weak solution to
the problem (\ref{3.3}) - (\ref{3.5}) if it satisfies the equality 
\begin{equation}
\frac{d}{dt}\left\langle u_{L},v\right\rangle _{\Omega _{L}}-\nu
\left\langle A_{1}u_{L},v\right\rangle _{\Omega _{L}}+\left\langle
B_{1}\left( u_{L},u_{L}\right) ,v\right\rangle _{\Omega _{L}}=\left\langle
f_{L},v\right\rangle _{\Omega _{L}},  \label{3.6}
\end{equation}%
for any $v\in V\left( \Omega _{L}\right) $ and almost everywhere in $\left(
0,T\right) $, and the initial condition 
\begin{equation}
\left\langle u_{L}\left( t\right) ,v\right\rangle \left\vert _{t=0}\right.
=\left\langle u_{0L},v\right\rangle ,\quad \forall v\in H\left( \Omega
_{L}\right) ,  \label{3.6b}
\end{equation}%
where $\left\langle \circ ,\circ \right\rangle _{\Omega _{L}}$ is the dual
form for the pair of spaces $\left( V\left( \Omega _{L}\right) ,V^{\ast
}\left( \Omega _{L}\right) \right) $. Here the space $\mathcal{V}\left(
Q_{L}^{T}\right) $ is defined as 
\begin{equation*}
\mathcal{V}\left( Q_{L}^{T}\right) \equiv \left\{ w\left\vert \ w\in V\left(
Q_{L}^{T}\right) ,\ w^{\prime }\in L^{2}\left( 0,T;V^{\ast }\left( \Omega
_{L}\right) \right) \right. \right\} .
\end{equation*}

So, to obtain the a priori estimates for possible weak solutions to the
problem, we will apply the usual approach. Then from (\ref{3.6}), we get 
\begin{equation}
\left\langle \frac{d}{dt}u_{L},u_{L}\right\rangle _{\Omega _{L}}-\nu
\left\langle A_{1}u_{L},u_{L}\right\rangle _{\Omega _{L}}+\left\langle
B_{1}\left( u_{L},u_{L}\right) ,u_{L}\right\rangle _{\Omega
_{L}}=\left\langle f_{L},u_{L}\right\rangle _{\Omega _{L}}.  \label{3.7}
\end{equation}%
Thence, by making the known calculations, taking into account the condition
on $\Omega _{L}$ and (\ref{3.4}), and also of calculations (\ref{3.1}) that
carried out in the previous section, we derive 
\begin{equation*}
\frac{1}{2}\frac{d}{dt}\left\Vert u_{L}\right\Vert _{H\left( \Omega
_{L}\right) }^{2}+\nu \left\Vert D_{1}u_{L}\right\Vert _{H\left( \Omega
_{L}\right) }^{2}+\left\Vert D_{2}u_{L}\right\Vert _{H\left( \Omega
_{L}\right) }^{2}+\frac{1}{4}\left\Vert \left(
a_{1}^{-1}D_{1}+a_{2}^{-1}D_{2}\right) u_{L}\right\Vert _{H\left( \Omega
_{L}\right) }^{2}+
\end{equation*}%
\begin{equation}
+\underset{i=1}{\overset{3}{\sum }}\underset{P_{x_{3}}\Omega _{L}}{\int }%
\left[ u_{L1}D_{1}+u_{L2}D_{2}-\frac{1}{2}u_{L3}\left(
a_{1}^{-1}D_{1}+a_{2}^{-1}D_{2}\right) \right] u_{Li},u_{Li}dx_{1}dx_{2}=%
\left\langle f_{L},u_{L}\right\rangle _{\Omega _{L}},  \label{3.8}
\end{equation}%
where $\left\langle g,h\right\rangle _{\Omega _{L}}=\underset{i=1}{\overset{3%
}{\sum }}\underset{P_{x_{3}}\Omega _{L}}{\int }g_{i}h_{i}dx_{1}dx_{2}$ for
any $g,h\in H\left( \Omega _{L}\right) $, or $g\in \left( W^{1,2}\left(
\Omega _{L}\right) \right) ^{3}$ and $h\in \left( W^{-1,2}\left( \Omega
_{L}\right) \right) ^{3}$, respectively.

So, providing similar calculations of such a type as were brought in (\ref%
{3.6a}), from equation (\ref{3.8}) we get the equation 
\begin{equation*}
\frac{1}{2}\frac{d}{dt}\left\Vert u_{L}\right\Vert _{H\left( \Omega
_{L}\right) }^{2}+\nu \left\Vert D_{1}u_{L}\right\Vert _{H\left( \Omega
_{L}\right) }^{2}+\left\Vert D_{2}u_{L}\right\Vert _{H\left( \Omega
_{L}\right) }^{2}+
\end{equation*}%
\begin{equation}
+\frac{\nu }{4}\left\Vert \left( a_{1}^{-1}D_{1}+a_{2}^{-1}D_{2}\right)
u_{L}\right\Vert _{H\left( \Omega _{L}\right) }^{2}=\left\langle
f_{L},u_{L}\right\rangle _{\Omega _{L}}  \label{3.9}
\end{equation}

Then, from the (\ref{3.8}), in view to (\ref{3.9}) and to (\ref{3.6b}), is
derived the following inequality 
\begin{equation*}
\frac{1}{2}\left\Vert u_{L}\right\Vert _{H\left( \Omega _{L}\right)
}^{2}\left( t\right) +\nu \underset{0}{\overset{t}{\tint }}\left\{ \underset{%
i=1}{\overset{3}{\sum }}\underset{P_{x_{3}}\Omega _{L}}{\int }\left[ \left(
D_{1}u_{Li}\right) ^{2}+\left( D_{2}u_{Li}\right) ^{2}\right] \right\}
dx_{1}dx_{2}ds+
\end{equation*}%
\begin{equation*}
+\frac{\nu }{4}\underset{0}{\overset{t}{\tint }}\left\{ \underset{i=1}{%
\overset{3}{\sum }}\underset{P_{x_{3}}\Omega _{L}}{\int }\left[ \left(
a_{1}^{-1}D_{1}+a_{2}^{-1}D_{2}\right) u_{L}\right] ^{2}dx_{1}dx_{2}\right\}
ds\leq
\end{equation*}%
\begin{equation}
\leq \underset{0}{\overset{t}{\tint }}\left\{ \underset{P_{x_{3}}\Omega _{L}}%
{\int }\left\vert \left( f_{L}\cdot u_{L}\right) \right\vert
dx_{1}dx_{2}\right\} ds+\frac{1}{2}\left\Vert u_{L0}\right\Vert _{H\left(
\Omega _{L}\right) }^{2},  \label{3.10}
\end{equation}%
which gives us the following a priori estimates 
\begin{equation}
\left\Vert u_{L}\right\Vert _{H\left( \Omega _{L}\right) }\left( t\right)
\leq C\left( f_{L},u_{L0},mes\Omega \right) ,  \label{3.11}
\end{equation}%
\begin{equation}
\left\Vert D_{1}u_{L}\right\Vert _{L_{2}\left( 0,T;H\left( \Omega
_{L}\right) \right) }+\left\Vert D_{2}u_{L}\right\Vert _{L_{2}\left(
0,T;H\left( \Omega _{L}\right) \right) }\leq C\left( f_{L},u_{L0},mes\Omega
\right) ,  \label{3.12}
\end{equation}%
where $C\left( f_{L},u_{L0},mes\Omega \right) >0$ is the independent of $%
u_{L}$ constant. Consequently, any possible solution to this problem belongs
to a bounded subset of the space $V\left( Q_{L}^{T}\right) $.

Boundedness of the trilinear form $\left( B_{1}\left( u_{L},u_{L}\right)
,v\right) $ from (\ref{3.7}) follows from the next result.

\begin{proposition}
\label{P_3.1}Let $u_{L}\in V\left( Q_{L}^{T}\right) $, $v\in V\left( \Omega
_{L}\right) $ and $B_{1}$ is the operator defined by 
\begin{equation*}
\left\langle B_{1}\left( u_{L},u_{L}\right) ,v\right\rangle _{\Omega
_{L}}=b_{L}\left( u_{L},u_{L},v\right) =
\end{equation*}%
\begin{equation*}
=\underset{i=1}{\overset{3}{\sum }}\underset{P_{x_{3}}\Omega _{L}}{\int }%
\left[ u_{L1}D_{1}+u_{L2}D_{2}-\frac{1}{2}u_{L3}\left(
a_{1}^{-1}D_{1}+a_{2}^{-1}D_{2}\right) \right] u_{Li},v_{Li}dx_{1}dx_{2}
\end{equation*}%
then $B_{1}\left( u_{L},u_{L}\right) $ belongs to bounded subset of $%
L^{2}\left( 0,T;V^{\ast }\left( \Omega _{L}\right) \right) $.
\end{proposition}

\begin{proof}
At first, we will show the boundedness of the operator $B_{1}$ acting from $%
V\left( \Omega _{L}\right) \times V\left( \Omega _{L}\right) $ to $V^{\ast
}\left( \Omega _{L}\right) $ for a. e. $t\in \left( 0,T\right) $. We have 
\begin{equation*}
\left\langle B_{1}\left( u_{L},u_{L}\right) ,v\right\rangle _{\Omega _{L}}=
\end{equation*}%
\begin{equation*}
=\underset{i=1}{\overset{3}{\sum }}\underset{P_{x_{3}}\Omega _{L}}{\int }%
\left[ u_{L1}D_{1}+u_{L2}D_{2}-\frac{1}{2}u_{L3}\left(
a_{1}^{-1}D_{1}+a_{2}^{-1}D_{2}\right) \right] u_{Li},v_{Li}dx_{1}dx_{2}=
\end{equation*}%
\begin{equation}
\underset{i=1}{\overset{3}{\sum }}\underset{P_{x_{3}}\Omega _{L}}{\int }%
\left[ \left( u_{L1}-\frac{1}{2}a_{1}^{-1}u_{L3}\right) D_{1}+\left( u_{L2}-%
\frac{1}{2}a_{2}^{-1}u_{L3}\right) D_{2}\right] u_{Li}v_{i}dx_{1}dx_{2}
\label{3.13}
\end{equation}%
due to (\ref{3.4}) .

Hence follows 
\begin{equation*}
\left\vert \left\langle B_{1}\left( u_{L},u_{L}\right) ,v\right\rangle
_{\Omega _{L}}\right\vert \leq c_{1}\underset{i=1}{\overset{3}{\sum }}%
\underset{P_{x_{3}}\Omega _{L}}{\int }\left\vert u_{L}\right\vert ^{2}\left(
\left\vert D_{1}v_{i}\right\vert +\left\vert D_{2}v_{i}\right\vert \right)
dx_{1}dx_{2}\leq
\end{equation*}%
\begin{equation}
\leq c\left\Vert u_{L}\right\Vert _{L^{4}\left( \Omega _{L}\right)
}^{2}\left\Vert v\right\Vert _{V\left( \Omega _{L}\right) }\Longrightarrow
\left\Vert B_{1}\left( u_{L},u_{L}\right) \right\Vert _{V^{\ast }\left(
\Omega _{L}\right) }\leq c\left\Vert u_{L}\right\Vert _{V\left( \Omega
_{L}\right) }^{2},  \label{3.14}
\end{equation}%
due to $V\left( \Omega _{L}\right) \subset L^{4}\left( \Omega _{L}\right) $.
This shows that operator $B_{1}:V\left( \Omega _{L}\right) \longrightarrow
V^{\ast }\left( \Omega _{L}\right) $ is bounded, and continuous for a. e. $%
t\in \left( 0,T\right) $, due to $V\left( \Omega _{L}\right) \subset
L^{4}\left( \Omega _{L}\right) $.

Finally, we obtain the needed result using the above inequality and the
well-known inequality 
\begin{equation*}
\overset{T}{\underset{0}{\int }}\left\Vert B_{1}\left( u_{L}\left( t\right)
,u_{L}\left( t\right) \right) \right\Vert _{V^{\ast }}^{2}dt\leq c\overset{T}%
{\underset{0}{\int }}\left\Vert u_{L}\left( t\right) \right\Vert
_{L^{4}}^{4}dt\leq c_{1}\overset{T}{\underset{0}{\int }}\left\Vert
u_{L}\left( t\right) \right\Vert _{H}^{2}\left\Vert u_{L}\right\Vert
_{V}^{2}dt\leq
\end{equation*}%
\begin{equation*}
\leq c_{1}\left\Vert u_{L}\right\Vert _{L^{\infty }\left( 0,T;H\right) }^{2}%
\overset{T}{\underset{0}{\int }}\left\Vert u_{L}\right\Vert
_{V}^{2}dt\Longrightarrow
\end{equation*}%
\begin{equation}
\left\Vert B_{1}\left( u_{L},u_{L}\right) \right\Vert _{L^{2}\left(
0,T;V^{\ast }\right) }\leq c_{1}\left\Vert u_{L}\right\Vert _{L^{\infty
}\left( 0,T;H\right) }\left\Vert u_{L}\right\Vert _{L^{2}\left( 0,T;V\right)
}.  \label{3.15}
\end{equation}

according to 2-dimensionality of the space variable $x$ (see, \cite{Lad1}, 
\cite{Lio1}, \cite{Tem1}, etc.). Thus, is proved that 
\begin{equation*}
B_{1}:V\left( Q_{L}^{T}\right) \times V\left( Q_{L}^{T}\right)
\longrightarrow L^{2}\left( 0,T;V^{\ast }\right)
\end{equation*}%
is a bounded operator, consequently, trilinear form $b_{L}\left(
u_{L},u_{L},v\right) $ is a bounded form also.
\end{proof}

So, it remains to receive the necessary a priori estimate for $\frac{%
\partial u_{L}}{\partial t}$ and the weak compactness of operator $%
B_{1}:V\left( \Omega _{L}\right) \longrightarrow V^{\ast }\left( \Omega
_{L}\right) $. \footnote{%
It should be noted that the obtained results are independent of choosing $2$
variables from $(x_{1},x_{2},x_{3})$, using which the $%
\Omega
_{L}$ is represented.}

For investigation, we will use Faedo-Galerkin's method since $V\left( \Omega
_{L}\right) $ is the separable space. It is not difficult to see, in this
case, the approximate solutions will have estimates analogous to (\ref{3.11}%
), (\ref{3.12}) and (\ref{3.15}). So, let $\left\{ w_{i}\right\}
_{i=1}^{\infty }$ is total system in $V\left( \Omega _{L}\right) $, then
from the system of equations (\ref{3.6})-(\ref{3.6b}) can define for each $%
m=1,2,...$ the system of equations 
\begin{equation*}
\left\langle \frac{d}{dt}u_{Lm},w_{j}\right\rangle _{\Omega
_{L}}=\left\langle \nu A_{1}u_{Lm},w_{j}\right\rangle _{\Omega
_{L}}+\left\langle B_{1}\left( u_{Lm},u_{Lm}\right) ,w_{j}\right\rangle
_{\Omega _{L}}+
\end{equation*}%
\begin{equation}
+\left\langle f_{L},w_{j}\right\rangle _{\Omega _{L}},\quad t\in \left( 0,T 
\right] ,\quad j=\overline{1,m},\quad  \label{3.16a}
\end{equation}%
\begin{equation*}
u_{Lm}\left( 0\right) =u_{0Lm}.
\end{equation*}%
for the approximate solutions $u_{Lm}$ (\ref{3.6})-(\ref{3.6b}) in the
following form

\begin{equation}
u_{Lm}=\overset{m}{\underset{i=1}{\sum }}u_{Lm}^{i}\left( t\right)
w_{i},\quad m=1,\ 2,....,  \label{3.16b}
\end{equation}%
where $u_{Lm}^{i}\left( t\right) $, $i=\overline{1,m}$ are unknown
functions, which will be determined by the (\ref{3.16a}). Here we assume $%
\left\{ u_{0Lm}\right\} _{m=1}^{\infty }\subset H\left( \Omega _{L}\right) $
be such sequence that $u_{0Lm}\longrightarrow u_{0L}$ in $H\left( \Omega
_{L}\right) $ as $m\longrightarrow \infty $. (Since $V\left( \Omega
_{L}\right) $ is everywhere dense in $H\left( \Omega _{L}\right) $, $u_{0Lm}$
can determine by using the total system $\left\{ w_{i}\right\}
_{i=1}^{\infty }$). \ 

So, using (\ref{3.16b}) in (\ref{3.16a}) we obtain 
\begin{equation*}
\overset{m}{\underset{j=1}{\sum }}\left\langle w_{j},w_{i}\right\rangle
_{\Omega _{L}}\frac{d}{dt}u_{Lm}^{j}\left( t\right) -\nu \overset{m}{%
\underset{j=1}{\sum }}\left\langle \Delta w_{j},w_{i}\right\rangle _{\Omega
_{L}}u_{Lm}^{j}\left( t\right) +
\end{equation*}%
\begin{equation*}
+\overset{m}{\underset{j,k=1}{\sum }}b_{L}\left( w_{j},w_{k},w_{i}\right)
u_{Lm}^{j}\left( t\right) u_{Lm}^{k}\left( t\right) =\left\langle
f_{L}\left( t\right) ,w_{i}\right\rangle _{\Omega _{L}},\ i=\overline{1,m}.
\end{equation*}%
As the matrix generated by $\left\langle w_{i},w_{j}\right\rangle _{\Omega
_{L}}$, $i,j=\overline{1,m}$, is nonsingular, therefore, its inverse exists.
Thanks to this, for unknown functions $u_{Lm}^{i}\left( t\right) $, $%
i=1,...,m$ we obtain the following Cauchy problem from the previous
equations 
\begin{equation*}
\frac{du_{Lm}^{i}\left( t\right) }{dt}=\overset{m}{\underset{j=1}{\sum }}%
c_{i,j}\left\langle f_{L}\left( t\right) ,w_{j}\right\rangle _{\Omega
_{L}}-\nu \overset{m}{\underset{j=1}{\sum }}d_{i,j}u_{Lm}^{j}\left( t\right)
+
\end{equation*}%
\begin{equation}
\overset{m}{\underset{j,k=1}{\sum }}h_{ijk}u_{Lm}^{j}\left( t\right)
u_{Lm}^{k}\left( t\right) ,  \label{3.16c}
\end{equation}%
\begin{equation*}
u_{Lm}^{i}\left( 0\right) =u_{0Lm}^{i},\quad i=1,...,m,\quad m=1,\ 2,...,
\end{equation*}%
where $u_{0Lm}^{i}$ is the $i^{th}$ component of $u_{0L}$ in this
representation $u_{0L}=\overset{\infty }{\underset{k=1}{\sum }}%
u_{0Lm}^{k}w_{k}$.

The Cauchy problem for the system of the nonlinear ordinary differential
equations (\ref{3.16c}) has a solution, which is defined on the whole of
interval $(0,T]$ due to uniformity of estimations received in (\ref{3.11}), (%
\ref{3.12}), and (\ref{3.15}). Consequently, the approximate solution $%
u_{Lm} $ exists and belong to a bounded subset of $W^{1,2}\left( 0,T;V^{\ast
}\left( \Omega _{L}\right) \right) $ for every $m=1,\ 2,...$ since the right
side of ((\ref{3.16c}) belong to a bounded subset of $L^{2}\left(
0,T;V^{\ast }\left( \Omega _{L}\right) \right) $ as were proved above (see, (%
\ref{3.11}), (\ref{3.12}), (\ref{3.15})).

It isn%
\'{}%
t difficult to see that if we take $\forall v\in V\left( \Omega _{L}\right) $
instead of $w_{k}$, and pass to limit concerning $m\longrightarrow \infty $
in equation (\ref{3.16b}) (could be, concerning a subsequence $\left\{
u_{Lm_{\mathit{l}}}\right\} _{\mathit{l}=1}^{\infty }$ of this sequence,
since it is known that such subsequence exists) \ 
\begin{equation}
\left\langle \frac{d}{dt}u_{L},v\right\rangle _{\Omega _{L}}=\left\langle
f_{L}+\nu \Delta u_{L}-\chi ,v\right\rangle _{\Omega _{L}},  \label{3.16d}
\end{equation}%
due to the fullness of the class $\left\{ w_{i}\right\} _{i=1}^{\infty }$ in 
$V\left( \Omega _{L}\right) .$ Where the function $\chi $ belongs to $%
L^{2}\left( 0,T;V^{\ast }\left( \Omega _{L}\right) \right) $ and is
determined by equality 
\begin{equation*}
\underset{\mathit{l}\longrightarrow \infty }{\lim }\left\langle B\left(
u_{Lm_{\mathit{l}}}\right) ,v\right\rangle _{\Omega _{L}}=\left\langle \chi
,v\right\rangle _{\Omega _{L}},
\end{equation*}%
which is shown in the above section. Hence, since in the equation (\ref%
{3.16d}) the right side belongs to $L^{2}\left( 0,T\right) $, therefore, and
the left side also belongs to $L^{2}\left( 0,T\right) $, according to the
previous a priori estimations, and due to Proposition \ref{P_3.1}, i.e. 
\begin{equation*}
\frac{du_{L}}{dt}\in L^{2}\left( 0,T;V^{\ast }\left( \Omega _{L}\right)
\right) .
\end{equation*}%
Consequently, the following result is proven.

\begin{proposition}
\label{P_3.2}Under conditions the Theorem \ref{Th_1}, $\frac{du_{L}}{dt}$
belongs to a bounded subset of the space $L^{2}\left( 0,T;V^{\ast }\left(
\Omega _{L}\right) \right) $.
\end{proposition}

From the above results of this section due to the abstract form of the
Riesz-Fischer theorem, the following statement holds

\begin{corollary}
\label{C_3.2}Under the above-mentioned conditions the function $u_{L}$
belongs to a bounded subset of the space $\mathcal{V}\left( Q_{L}^{T}\right) 
$, where 
\begin{equation}
\mathcal{V}\left( Q_{L}^{T}\right) \equiv V\left( Q_{L}^{T}\right) \cap
W^{1,2}\left( 0,T;V^{\ast }\left( \Omega _{L}\right) \right) .  \label{3.16}
\end{equation}
\end{corollary}

Thus, for the proof that $u_{L}$ is the solution to the problem (\ref{3.3})
- (\ref{3.5}) (or the problem (\ref{3.6})-(\ref{3.6b})) remains to shows
that $\chi =B_{1}\left( u_{L},u_{L}\right) $ (or $\left\langle \chi
,v\right\rangle _{\Omega _{L}}=\left( B_{1}\left( u_{L},u_{L}\right)
,v\right) $) for $\forall v\in V\left( \Omega _{L}\right) $. \ \ 

\begin{proposition}
\label{P_3.3}Operator $B_{1}:V\left( Q_{L}^{T}\right) \longrightarrow
L^{2}\left( 0,T;V^{\ast }\left( \Omega _{L}\right) \right) $ is a weakly
compact operator, in other words, if the sequence $\left\{ u_{L}^{m}\right\}
_{1}^{\infty }\subset V\left( Q_{L}^{T}\right) $ is weakly converged in $%
V\left( Q_{L}^{T}\right) $ then there exists, at least, such subsequence $%
\left\{ u_{L}^{m_{k}}\right\} _{1}^{\infty }\subset \left\{
u_{L}^{m}\right\} _{1}^{\infty }$ that the sequence $\left\{ B_{1}\left(
u_{L}^{m_{k}},u_{L}^{m_{k}}\right) \right\} _{1}^{\infty }$ weakly converge
in $L^{2}\left( 0,T;V^{\ast }\left( \Omega _{L}\right) \right) $.
\end{proposition}

\begin{proof}
Let the sequence $\left\{ u_{L}^{m}\right\} _{1}^{\infty }\subset V\left(
Q_{L}^{T}\right) $ is weakly converged to $u_{L}^{0}$ in $V\left(
Q_{L}^{T}\right) $. Then there exists such subsequence $\left\{
u_{L}^{m_{k}}\right\} _{1}^{\infty }\subset \left\{ u_{L}^{m}\right\}
_{1}^{\infty }$ that $u_{L}^{m_{k}}\longrightarrow u_{L}^{0}$ in $%
L^{2}\left( 0,T;H\right) $, due to the known theorems on the compactness of
the embedding, particularly, as is known the following embedding 
\begin{equation*}
\mathcal{V}\left( Q_{L}^{T}\right) \equiv L^{2}\left( 0,T;V\left( \Omega
_{L}\right) \right) \cap W^{1,2}\left( 0,T;V^{\ast }\left( \Omega
_{L}\right) \right) \subset L^{2}\left( 0,T;H\right)
\end{equation*}%
is compact (see, e. g. \cite{Lio1}, \cite{Tem1}, \cite{Sol3}).

It is enough to show that generated by the expression 
\begin{equation*}
\underset{i=1}{\overset{3}{\sum }}\left[ \left( u_{L1}-\frac{1}{2}%
a_{1}^{-1}u_{L3}\right) D_{1}+\left( u_{L2}-\frac{1}{2}a_{2}^{-1}u_{L3}%
\right) D_{2}\right] u_{Li}
\end{equation*}%
operator $B_{1}$ is weakly compact from $\mathcal{V}\left( Q_{L}^{T}\right) $
to the $L^{2}\left( 0,T;V^{\ast }\left( \Omega _{L}\right) \right) $.

The operator $B_{1}:\mathcal{V}\left( Q_{L}^{T}\right) \longrightarrow L^{%
\frac{4}{3}}\left( 0,T;V^{\ast }\left( \Omega _{L}\right) \right) $ is
bounded, (i.e. the image of each bounded subset from the space $\mathcal{V}%
\left( Q_{L}^{T}\right) $ under the mapping $B_{1}$ is a bounded subset of
the space $L^{2}\left( 0,T;V^{\ast }\left( \Omega _{L}\right) \right) $) by
a priori estimations, and Proposition \ref{P_3.1}.\ \ 

From mentioned compactness theorem follows the sequence $\left\{
u_{L}^{m}\right\} _{1}^{\infty }$ posses a subsequence $\left\{
u_{L}^{m_{k}}\right\} _{1}^{\infty }\subset \left\{ u_{L}^{m}\right\}
_{1}^{\infty }$ that strongly convergent in the space $L^{2}\left(
0,T;H\right) $ to some element $u_{L}$ of $L^{2}\left( 0,T;H\right) $.
Consequently, $B_{1}\left( \left\{ u_{L}^{m_{k}}\right\} _{1}^{\infty
},\left\{ u_{L}^{m_{k}}\right\} _{1}^{\infty }\right) $ belongs to a bounded
subset of space $L^{2}\left( 0,T;V^{\ast }\left( \Omega _{L}\right) \right) $%
. Thence lead that there is such element $\chi \in L^{2}\left( 0,T;V^{\ast
}\left( \Omega _{L}\right) \right) $ that sequence $B_{1}\left(
u_{L}^{m_{k}},u_{L}^{m_{k}}\right) $ weakly converges to $\chi $ when $%
m_{k}\nearrow \infty $, i.e. 
\begin{equation}
B_{1}\left( u_{L}^{m_{k}},u_{L}^{m_{k}}\right) \rightharpoonup \chi \quad 
\text{ in }L^{2}\left( 0,T;V^{\ast }\left( \Omega _{L}\right) \right) ,
\label{3.17}
\end{equation}%
due to the reflexivity of this space, there exists, at least, such
subsequence by which this occurs.

If we set the vector space 
\begin{equation*}
\mathcal{C}^{1}\left( \overline{Q}_{L}\right) \equiv \left\{ v\left\vert \
v_{i}\in C^{1}\left( \left[ 0,T\right] ;C_{0}^{1}\left( \overline{\Omega _{L}%
}\right) \right) ,\right. i=1,2,3\right\}
\end{equation*}%
and consider the trilinear form 
\begin{equation*}
\underset{0}{\overset{T}{\int }}\left\langle B\left( u_{L}^{m}\right)
,v\right\rangle _{\Omega _{L}}dt=\underset{0}{\overset{T}{\int }}b\left(
u_{L}^{m},u_{L}^{m},v\right) dt=\underset{0}{\overset{T}{\int }}\left\langle 
\underset{j=1}{\overset{3}{\sum }}u_{Lj}^{m}D_{j}u_{L}^{m},v\right\rangle
_{\Omega _{L}}dt=
\end{equation*}%
for $v\in \mathcal{C}^{1}\left( \overline{Q}_{L}\right) $, then we get 
\begin{equation}
-\underset{i=1}{\overset{3}{\sum }}\underset{0}{\overset{T}{\int }}\underset{%
P_{x_{3}}\Omega _{L}}{\int }\left[ \left( _{L1}^{m}-\frac{1}{2}%
a_{1}^{-1}u_{L3}^{m}\right) D_{1}v_{i}+\left( u_{L2}^{m}-\frac{1}{2}%
a_{2}^{-1}u_{L3}^{m}\right) D_{2}v_{i}\right] u_{Li}^{m}dx_{1}dx_{2}dt.
\label{3.17a}
\end{equation}

according to (\ref{3.13}). Now, if we separately take the arbitrary term of
this sum, then it isn't difficult to see that the following convergences are
true 
\begin{equation*}
u_{Li}^{m}u_{L1}^{m}\rightharpoonup u_{Li}u_{L1};\quad
a_{1}^{-1}u_{Li}^{m}u_{L3}^{m}\rightharpoonup a_{1}^{-1}u_{Li}u_{L3}
\end{equation*}%
since $u_{Li}^{m_{k}}\longrightarrow u_{Li}$ in $L^{2}\left( 0,T;H\right) $
and $u_{Li}^{m_{k}}\rightharpoonup u_{Li}$ in $L^{\infty }\left(
0,T;H\right) $ $\ast -$ weakly,\ as $\left\{ u_{L}^{m}\right\} _{1}^{\infty
} $ belongs to a bounded subset of $\mathcal{V}\left( Q_{L}^{T}\right) $,
and this is fulfilled for each term of (\ref{3.17a}).

Thus, passing to the limit when $m_{k}\nearrow \infty $ we obtain 
\begin{equation*}
\chi =B_{1}\left( u_{L},u_{L}\right) \Longrightarrow B_{1}\left(
u_{L}^{m_{k}},u_{L}^{m_{k}}\right) \rightharpoonup B_{1}\left(
u_{L},u_{L}\right) \text{ in the distribution sense.}
\end{equation*}%
Whence using the density of $\mathcal{C}^{1}\left( \overline{Q}_{L}\right) $
in $\mathcal{V}\left( Q_{L}^{T}\right) $, in addition, since $B_{1}\left(
u_{L}^{m_{k}},u_{L}^{m_{k}}\right) \rightharpoonup \chi $ takes place in the
space $L^{2}\left( 0,T;V^{\ast }\left( \Omega _{L}\right) \right) $ we get
that $\chi =B_{1}\left( u_{L},u_{L}\right) $ also takes place in this space.
\end{proof}

Consequently, we proved the existence of the function $u_{L}\in \mathcal{V}%
\left( Q_{L}^{T}\right) $, that satisfies equations (\ref{3.6})-(\ref{3.6b})
applying to this problem of the Faedo-Galerkin method, and using the
above-mentioned results.

Whence, one can conclude that function $u_{L}$ satisfies the problem (\ref%
{3.6})-(\ref{3.6b}) if we show the initial condition also satisfied in the
distribution sense. We will the proof of the satisfies of the initial
condition according to the usual way (see, \cite{Lad1}, \cite{Lio1}, \cite%
{Tem1} and references therein).

Let $\phi $ be a continuously differentiable function on $[0,T]$ with $\phi
(T)=0$. Multiplying (\ref{3.16b}) by $\phi (t)$, and integrating by parts in
the first integral we lead to the equation 
\begin{equation*}
-\underset{0}{\overset{T}{\int }}\left\langle u_{Lm},\frac{d}{dt}\phi
(t)w_{j}\right\rangle _{\Omega _{L}}dt=\underset{0}{\overset{T}{\int }}%
\left\langle \nu \Delta u_{Lm},\phi (t)w_{j}\right\rangle _{\Omega _{L}}dt+
\end{equation*}%
\begin{equation*}
\underset{0}{\overset{T}{\int }}b\left( u_{Lm},u_{Lm},\phi (t)w_{j}\right)
dt+\underset{0}{\overset{T}{\int }}\left\langle f_{L},\phi
(t)w_{j}\right\rangle _{\Omega _{L}}dt+\left\langle u_{0Lm},\phi
(0)w_{j}\right\rangle _{\Omega _{L}}.
\end{equation*}

Where can pass to the limit with respect to subsequence $\left\{
u_{Lm_{l}}\right\} _{l=1}^{\infty }$ of the sequence $\left\{ u_{Lm}\right\}
_{m=1}^{\infty }$ owing to the above-mentioned results. Then we arrive at
the equation 
\begin{equation*}
-\underset{0}{\overset{T}{\int }}\left\langle u_{L},\frac{d}{dt}\phi
(t)w_{j}\right\rangle _{\Omega _{L}}dt=\underset{0}{\overset{T}{\int }}%
\left\langle \nu \Delta u_{L},\phi (t)w_{j}\right\rangle _{\Omega _{L}}dt+
\end{equation*}%
\begin{equation}
\underset{0}{\overset{T}{\int }}b\left( u_{L},u_{L},\phi (t)w_{j}\right) dt+%
\underset{0}{\overset{T}{\int }}\left\langle f_{L},\phi
(t)w_{j}\right\rangle _{\Omega _{L}}dt+\left\langle u_{0L},\phi
(0)w_{j}\right\rangle _{\Omega _{L}},  \label{3.17b}
\end{equation}%
that holds for each $w_{j}$, $j=1,2,...$. Consequently, this equality holds
for any finite linear combination of the $w_{j}$, moreover, the (\ref{3.17b}%
) remains true also for any $v\in V\left( \Omega _{L}\right) $ due to the
continuability. Whence, one can conclude that function $u_{L}$ satisfies
equation (\ref{3.6}) in the distribution sense.

Now, if multiply (\ref{3.6}) by $\phi (t)$, and integrate concerning $t$,
then after integrating the first term by parts we get 
\begin{equation*}
-\underset{0}{\overset{T}{\int }}\left\langle u_{L},v\frac{d}{dt}\phi
(t)\right\rangle _{\Omega _{L}}dt-\underset{0}{\overset{T}{\int }}%
\left\langle \nu \Delta u_{L},\phi (t)v\right\rangle _{\Omega _{L}}dt+
\end{equation*}

\begin{equation*}
\underset{0}{\overset{T}{\int }}\left\langle \underset{j=1}{\overset{3}{\sum 
}}u_{Lj}D_{j}u_{L},\phi (t)v\right\rangle _{\Omega _{L}}dt=\underset{0}{%
\overset{T}{\int }}\left\langle f_{L},\phi (t)v\right\rangle _{\Omega
_{L}}dt+\left\langle u_{L}\left( 0\right) ,\phi (0)v\right\rangle _{\Omega
_{L}}.
\end{equation*}

If compare this with (\ref{3.17b}) after replacing $w_{j}$ with any $v\in
V\left( \Omega _{L}\right) $ then we obtain 
\begin{equation*}
\phi (0)\left\langle u_{L}\left( 0\right) -u_{0L},v\right\rangle _{\Omega
_{L}}=0.
\end{equation*}%
Thus, we get the satisfies of the initial condition due to arbitrariness of $%
v\in V\left( \Omega _{L}\right) $ and $\phi $, since function $\phi $ can be
chosen as $\phi (0)\neq 0$.

Consequently, the following result is proven.

\begin{theorem}
\label{Th_2.1}Under conditions the Theorem \ref{Th_1} for any functions $%
u_{0L}\in \left( H\left( \Omega _{L}\right) \right) ^{3}$ and $f_{L}\in
L^{2}\left( 0,T;V^{\ast }\left( \Omega _{L}\right) \right) $ the problem (%
\ref{3.3}) - (\ref{3.5}) has a weak solution $u_{L}\left( t,x\right) $, that
belongs to $\mathcal{V}\left( Q_{L}^{T}\right) $.
\end{theorem}

\begin{remark}
The obtained a priori estimations and Propositions \ref{P_3.1}, \ref{P_3.3}
show for the proof of the solvability of the problem (\ref{3.3}) - (\ref{3.5}%
) can be to use also the compactness method (see, e.g. \cite{Sol3}, \cite%
{SolAhm}, and also, \cite{Sol2}, \cite{Sol4}).
\end{remark}

\section{\label{Sec_5}Uniqueness of Weak Solution of Problem (\protect\ref%
{3.3}) - (\protect\ref{3.5})}

We will study the uniqueness of the solution as usual: assume that the posed
problem has, at least, two different solutions $u$ and $v$, and will show
that this is impossible.

To do so, it needs to study the following problem for the function $w=u-v$ 
\begin{equation*}
\frac{\partial w}{\partial t}-\nu \left[ \left( D_{1}^{2}+D_{2}^{2}\right) +%
\frac{1}{4}\left( a_{1}^{-1}D_{1}+a_{2}^{-1}D_{2}\right) ^{2}\right] w_{L}+
\end{equation*}%
\begin{equation*}
+\left[ u_{L1}D_{1}+u_{L2}D_{2}-\frac{1}{2}u_{L3}\left(
a_{1}^{-1}D_{1}+a_{2}^{-1}D_{2}\right) \right] u_{L}-
\end{equation*}%
\begin{equation*}
-\left[ v_{L1}D_{1}+v_{L2}D_{2}-\frac{1}{2}v_{L3}\left(
a_{1}^{-1}D_{1}+a_{2}^{-1}D_{2}\right) \right] v_{L}=0
\end{equation*}%
\begin{equation}
\func{div}w_{L}=D_{1}w_{L1}+D_{2}w_{L2}-\frac{1}{2}\left(
a_{1}^{-1}D_{1}+a_{2}^{-1}D_{2}\right) w_{L3}=0,  \label{3.18}
\end{equation}%
\begin{equation}
w_{L}\left( 0,x\right) =0,\quad x\in \Omega \cap L;\quad w_{L}\left\vert \
_{\left( 0,T\right) \times \partial \Omega _{L}}\right. =0.  \label{3.19}
\end{equation}

Hence we derive 
\begin{equation*}
\frac{1}{2}\frac{d}{dt}\left\Vert w\right\Vert _{2}^{2}+\nu \left(
\left\Vert D_{1}w\right\Vert _{2}^{2}+\left\Vert D_{2}w\right\Vert
_{2}^{2}\right) +\frac{\nu }{4}\left\Vert \left(
a_{1}^{-1}D_{1}+a_{2}^{-1}D_{2}\right) w_{L}\right\Vert _{H\left( \Omega
_{L}\right) }^{2}+
\end{equation*}%
\begin{equation*}
+\left\langle u_{L1}D_{1}u_{L}-v_{L1}D_{1}v_{L},w_{L}\right\rangle _{\Omega
_{L}}+\left\langle u_{L2}D_{2}u_{L}-v_{L2}D_{2}v_{L},w_{L}\right\rangle
_{\Omega _{L}}-
\end{equation*}%
\begin{equation}
\frac{1}{2}a_{1}^{-1}\left\langle
u_{L3}D_{1}u_{L}-v_{L3}D_{1}v_{L},w_{L}\right\rangle _{\Omega _{L}}-\frac{1}{%
2}a_{2}^{-1}\left\langle
u_{L3}D_{2}u_{L}-v_{L3}D_{2}v_{L},w_{L}\right\rangle _{\Omega _{L}}=0.
\label{3.20}
\end{equation}

If consider separately the last 4 terms in the sum of the left part of (\ref%
{3.20}) and these are simplified by calculations then we get 
\begin{equation*}
\left\langle w_{L1}D_{1}u_{L}+v_{L1}D_{1}w_{L},w_{L}\right\rangle _{\Omega
_{L}}+\left\langle w_{L2}D_{2}u_{L}+v_{L2}D_{2}w_{L},w_{L}\right\rangle
_{\Omega _{L}}-
\end{equation*}%
\begin{equation*}
-\frac{1}{2a_{1}}\left\langle
w_{L3}D_{1}u_{L}+v_{L3}D_{1}w_{L},w_{L}\right\rangle _{\Omega _{L}}-\frac{1}{%
2a_{2}}\left\langle w_{L3}D_{2}u_{L}-v_{L3}D_{2}w_{L},w_{L}\right\rangle
_{\Omega _{L}}=
\end{equation*}%
\begin{equation*}
=\left\langle w_{L1}D_{1}u_{L}+v_{L1}D_{1}w_{L},w_{L}\right\rangle _{\Omega
_{L}}+\left\langle w_{L2}D_{2}u_{L}+v_{L2}D_{2}w_{L},w_{L}\right\rangle
_{\Omega _{L}}-
\end{equation*}%
\begin{equation*}
-\frac{1}{2a_{1}}\left\langle
w_{L3}D_{1}u_{L}+v_{3}D_{1}w_{L},w_{L}\right\rangle _{\Omega _{L}}-\frac{1}{%
2a_{2}}\left\langle w_{3}D_{2}u_{L}+v_{3}D_{2}w_{L},w_{L}\right\rangle
_{\Omega _{L}}=
\end{equation*}%
\begin{equation*}
=\left\langle \left( w_{L1}D_{1}+w_{L2}D_{2}\right) u_{L},w_{L}\right\rangle
_{\Omega _{L}}+\frac{1}{2}\left\langle v_{L1},D_{1}w_{L}^{2}\right\rangle
_{\Omega _{L}}+\frac{1}{2}\left\langle v_{L2},D_{2}w_{L}^{2}\right\rangle
_{\Omega _{L}}-
\end{equation*}%
\begin{equation*}
\left\langle w_{L3}\left( \frac{1}{2a_{1}}D_{1}+\frac{1}{2a_{2}}D_{2}\right)
u_{L},w_{L}\right\rangle _{\Omega _{L}}-\left\langle \frac{v_{L3}}{4a_{1}}%
,D_{1}w_{L}^{2}\right\rangle _{\Omega _{L}}-\left\langle \frac{v_{L3}}{4a_{2}%
},D_{2}w_{L}^{2}\right\rangle _{\Omega _{L}}
\end{equation*}%
\begin{equation*}
=\frac{1}{2}\left\langle v_{L1}-\frac{1}{2a_{1}}v_{L3},D_{1}w^{2}\right%
\rangle _{\Omega _{L}}+\frac{1}{2}\left\langle v_{L2}-\frac{1}{2a_{2}}%
v_{L3},D_{2}w_{L}^{2}\right\rangle _{\Omega _{L}}+
\end{equation*}%
\begin{equation*}
+\left\langle \left( w_{L1}-\frac{1}{2a_{1}}w_{L3}\right)
D_{1}u_{L},w_{L}\right\rangle _{\Omega _{L}}+\left\langle \left( w_{L2}-%
\frac{1}{2a_{2}}w_{L3}\right) D_{2}u_{L},w_{L}\right\rangle _{\Omega _{L}}=
\end{equation*}%
\begin{equation*}
=\left\langle \left( w_{L1}-\frac{1}{2a_{1}}w_{L3}\right)
D_{1}u_{L},w_{L}\right\rangle _{\Omega _{L}}+\left\langle \left( w_{L2}-%
\frac{1}{2a_{2}}w_{L3}\right) D_{2}u_{L},w_{L}\right\rangle _{\Omega _{L}},
\end{equation*}%
where used the equation $\func{div}v=0$ (see, (\ref{3.4})) and the condition
(\ref{3.19}). Takes into account this equality in equation (\ref{3.20}) then
we get 
\begin{equation*}
\frac{1}{2}\frac{d}{dt}\left\Vert w_{L}\right\Vert _{2}^{2}+\nu \left(
\left\Vert D_{1}w_{L}\right\Vert _{2}^{2}+\left\Vert D_{2}w_{L}\right\Vert
_{2}^{2}\right) +\frac{\nu }{4}\left\Vert \left(
a_{1}^{-1}D_{1}+a_{2}^{-1}D_{2}\right) w_{L}\right\Vert _{H\left( \Omega
_{L}\right) }^{2}+
\end{equation*}%
\begin{equation}
+\left\langle \left( w_{1}-\frac{1}{2a_{1}}w_{3}\right)
D_{1}u_{L},w_{L}\right\rangle _{\Omega _{L}}+\left\langle \left( w_{2}-\frac{%
1}{2a_{2}}w_{3}\right) D_{2}u_{L},w_{L}\right\rangle _{\Omega _{L}}=0,
\label{3.21}
\end{equation}

for a.e. $t\in \left( 0,T\right) .$Thus, we arrive the Cauchy problem for
equation (\ref{3.21}) with the initial condition 
\begin{equation}
\left\Vert w_{L}\right\Vert _{2}\left( 0\right) =0.  \label{3.22}
\end{equation}

Hence, by conducting the appropriate estimates, we get the following Cauchy
problem for the differential inequation 
\begin{equation*}
\frac{1}{2}\frac{d}{dt}\left\Vert w_{L}\right\Vert _{2}^{2}+\nu \left(
\left\Vert D_{1}w_{L}\right\Vert _{2}^{2}+\left\Vert D_{2}w_{L}\right\Vert
_{2}^{2}\right) \leq
\end{equation*}%
\begin{equation}
\left\vert \left\langle \left( w_{L1}-a_{1}^{-1}w_{_{L}3}\right)
w_{L},D_{1}u_{L}\right\rangle _{\Omega _{L}}\right\vert +\left\vert
\left\langle \left( w_{L2}-a_{2}^{-1}w_{L3}\right)
w_{L},D_{2}u_{L}\right\rangle _{\Omega _{L}}\right\vert ,  \label{3.21a}
\end{equation}

with the initial condition (\ref{3.22}).

If to estimate the right side of (\ref{3.21a}) then we have 
\begin{equation*}
\left\vert \left\langle \left( w_{L1}-a_{1}^{-1}w_{L3}\right)
w_{L},D_{1}u_{L}\right\rangle _{\Omega _{L}}\right\vert +\left\vert
\left\langle \left( w_{L2}-a_{2}^{-1}w_{L3}\right)
w_{L},D_{2}u_{L}\right\rangle _{\Omega _{L}}\right\vert \leq
\end{equation*}%
\begin{equation*}
\left( \left\Vert w_{L1}-a_{1}^{-1}w_{L3}\right\Vert _{4}+\left\Vert
w_{L2}-a_{2}^{-1}w_{L3}\right\Vert _{4}\right) \left\Vert w_{L}\right\Vert
_{4}\left\Vert \nabla u_{L}\right\Vert _{2}\leq
\end{equation*}%
whence follows 
\begin{equation*}
\left( 1+\max \left\{ \left\vert a_{1}^{-1}\right\vert ,\left\vert
a_{2}^{-1}\right\vert \right\} \right) \left\Vert w_{L}\right\Vert
_{4}^{2}\left\Vert \nabla u_{L}\right\Vert _{2}\leq c\left\Vert
w_{L}\right\Vert _{2}\left\Vert \nabla w_{L}\right\Vert _{2}\left\Vert
\nabla u_{L}\right\Vert _{2}
\end{equation*}%
thanks of Gagliardo-Nirenberg inequality (\cite{BesIlNik}).

It needs to be noted that 
\begin{equation*}
\left( w_{L1}-a_{1}^{-1}w_{L3}\right) w_{L},\ \left(
w_{L2}-a_{2}^{-1}w_{L3}\right) w_{L}\in L^{2}\left( 0,T;V^{\ast }\left(
\Omega _{L}\right) \right) ,
\end{equation*}%
due to (\ref{3.16}).

Now, taking into account this in (\ref{3.21a}) one can arrive at the
following Cauchy problem for differential inequation 
\begin{equation*}
\frac{1}{2}\frac{d}{dt}\left\Vert w_{L}\right\Vert _{2}^{2}\left( t\right)
+\nu \left\Vert \nabla w_{L}\right\Vert _{2}^{2}\left( t\right) \leq
c\left\Vert w_{L}\right\Vert _{2}\left( t\right) \left\Vert \nabla
w_{L}\right\Vert _{2}\left( t\right) \left\Vert \nabla u_{L}\right\Vert
_{2}\left( t\right) \leq
\end{equation*}%
\begin{equation*}
C\left( c,\nu \right) \left\Vert \nabla u_{L}\right\Vert _{2}^{2}\left(
t\right) \left\Vert w_{L}\right\Vert _{2}^{2}\left( t\right) +\nu \left\Vert
\nabla w_{L}\right\Vert _{2}^{2}\left( t\right) ,\quad \left\Vert
w_{L}\right\Vert _{2}\left( 0\right) =0,
\end{equation*}%
since $w_{L}\in L^{\infty }\left( 0,T;H\left( \Omega _{L}\right) \right) $,
where $C\left( c,\nu \right) >0$ is constant. Consequently, according to the
existence theorem, $w_{L}\in \mathcal{V}\left( Q_{L}^{T}\right) $ and $%
\left\Vert w_{L}\right\Vert _{2}\left\Vert \nabla w_{L}\right\Vert _{2}\in
L^{2}\left( 0,T\right) $.

Thus, we obtain the problem 
\begin{equation*}
\frac{d}{dt}\left\Vert w_{L}\right\Vert _{2}^{2}\left( t\right) \leq
2C\left( c,\nu \right) \left\Vert \nabla u_{L}\right\Vert _{2}^{2}\left(
t\right) \left\Vert w_{L}\right\Vert _{2}^{2}\left( t\right) ,\quad
\left\Vert w_{L}\right\Vert _{2}\left( 0\right) =0.
\end{equation*}%
If to denote $\left\Vert w_{L}\right\Vert _{2}^{2}\left( t\right) \equiv
y\left( t\right) $ then 
\begin{equation*}
\frac{d}{dt}y\left( t\right) \leq 2C\left( c,\nu \right) \left\Vert \nabla
u_{L}\right\Vert _{2}^{2}\left( t\right) y\left( t\right) ,\quad y\left(
0\right) =0.
\end{equation*}

Whence follows $\left\Vert w\right\Vert _{L2}^{2}\left( t\right) \equiv
y\left( t\right) =0$, and consequently the following result is proven. \ 

\begin{theorem}
\label{Th_2.2}Under conditions the Theorem \ref{Th_1} for any 
\begin{equation*}
\left( f_{L},u_{L0}\right) \in L^{2}\left( 0,T;V^{\ast }\left( \Omega
_{L}\right) \right) \times H\left( \Omega _{L}\right)
\end{equation*}%
the problem (\ref{3.3}) - (\ref{3.5}) has a unique weak solution $%
u_{L}\left( t,x\right) $ that is contained in $\mathcal{V}\left(
Q_{L}^{T}\right) $.
\end{theorem}

\section{\label{Sec_6}Proof of Theorem \protect\ref{Th_1}}

\begin{proof}
(of Theorem \ref{Th_1}). It is well known, under the above-mentioned
conditions problem (\ref{2.1}) - (\ref{2.3}) is weakly solvable and any
solution belongs to the space $\mathcal{V}\left( Q^{T}\right) $ (see, Section%
\ref{Sec_2}). Consequently, under the conditions of Theorem \ref{Th_1} this
problem has a weak solution that belongs, at least, to the space $\mathcal{V}%
\left( Q^{T}\right) $. Since the auxiliary problems (\ref{3.3})-(\ref{3.5})
are constructed in such a way that each of these is the reduced (projective)
problem of the main problem (\ref{2.1}) - (\ref{2.3}), consequently, their
solutions are the projections of the solution to the main problem onto the
cross-section of the domain. Therefore, under the conditions of Theorem \ref%
{Th_1} the relations between the weak solution to the main problem (\ref{2.1}%
) - (\ref{2.3}) and the weak solution to the auxiliary problems (\ref{3.3})-(%
\ref{3.5}) satisfy the conditions of Lemma \ref{L_2.2}. Consequently,
according to the above Lemma \ref{L_2.2} it is sufficient to investigate the
existence and uniqueness of the weak solution to the posed problem on an
arbitrarily fixed subdomain to show the existence or non-existence of such
subdomains that, on which the studied problem can possess more than one
solution. In other words, it is sufficient to study this question in the
case when subdomains are generated by arbitrarily fixed hyperplanes by Lemma %
\ref{L_2.2}. So to do this, it is enough to prove that such subdomains don't
exist. Whence follows, to end the proof, it needs to use the obtained above
results on the existence and uniqueness of the weak solution to the
auxiliary problems (\ref{3.3})-(\ref{3.5}) (i.e. Theorems \ref{Th_2.1} and %
\ref{Th_2.2}).

Consequently, we can to employ Lemma \ref{L_2.2} to these functions $u\in 
\mathcal{V}\left( Q^{T}\right) $ on $Q^{T}$ and $u_{L}\in \mathcal{V}\left(
Q_{L}^{T}\right) $ on $Q_{L}^{T}$, respectively. So, assume problem (\ref%
{2.1}) - (\ref{2.3}) has, at least, two different weak solutions under
conditions of Theorem \ref{Th_1}. Then there exists, at least, such
subdomain of $Q^{T}\equiv \left( 0,T\right) \times \Omega $, on which these
functions are different.

But, as follows from theorems that were proved for the problem (\ref{3.3})-(%
\ref{3.5}) in the previous sections do not exist such subdomains, on which
the problem (\ref{2.1}) - (\ref{2.3}) could be possessed more than one weak
solution according to Lemma \ref{L_2.2}.

Consequently, the uniqueness of the weak solution to the problem (\ref{2.1})
- (\ref{2.3}) under the conditions of Theorem \ref{Th_1} proved thanks to
Lemma \ref{L_2.2}.
\end{proof}

Hence can make the following conclusion.

\subsection{Conclusion I}

Since $L^{2}\left( 0,T;H^{1/2}\left( \Omega \right) \right) $ and $%
H^{1/2}\left( \Omega \right) $ are everywhere dense in spaces $L^{2}\left(
0,T;V^{\ast }\left( \Omega \right) \right) $ and $H\left( \Omega \right) $,
respectively, then any neighborhoods of the given functions from $\qquad
L^{2}\left( 0,T;V^{\ast }\left( \Omega \right) \right) $ and $H\left( \Omega
\right) $ contains elements from $L^{2}\left( 0,T;H^{1/2}\left( \Omega
\right) \right) $ and $H^{1/2}\left( \Omega \right) $, respectively.

Then one can state that there exists such everywhere dense subset $U$ of $%
H\left( \Omega \right) $ $U:$ $H^{1/2}\left( \Omega \right) \subseteq
U\subset H\left( \Omega \right) $ and such everywhere dense subset $V$ of $%
L^{2}\left( 0,T;V^{\ast }\left( \Omega \right) \right) $ 
\begin{equation*}
V:L^{2}\left( 0,T;H^{1/2}\left( \Omega \right) \right) \subseteq V\subset
L^{2}\left( 0,T;V^{\ast }\left( \Omega \right) \right)
\end{equation*}%
that for each $u_{0}\in U$ and $f\in V$ the problem (\ref{2.1}) - (\ref{2.3}%
) has a unique weak solution.

\section{\label{Sec_7}One "local" result on the uniqueness of the weak
solution}

We believe there has a sense to provide here still one result that is
connected with the same question, which is without complementary smoothness
conditions. Here, the known approach that takes into account other
properties of this problem is used.

Let the posed problem (\ref{2.1})-(\ref{2.3}) have two different solutions: $%
u,v\in \mathcal{V}\left( Q^{T}\right) $, then within the known approach we
get the following problem for vector function $w(t,x)=u(t,x)-v(t,x)$ 
\begin{equation}
\frac{1}{2}\frac{\partial }{\partial t}\left\Vert w\right\Vert _{2}^{2}+\nu
\left\Vert \nabla w\right\Vert _{2}^{2}+\underset{j,k=1}{\overset{3}{\sum }}%
\left\langle \frac{\partial v_{k}}{\partial x_{j}}w_{k},w_{j}\right\rangle
=0,  \label{7.1}
\end{equation}%
\begin{equation}
w\left( 0,x\right) =w_{0}\left( x\right) =0,\quad x\in \Omega ;\quad
w\left\vert \ _{\left[ 0,T\right] \times \partial \Omega }=0\right. ,
\label{7.2}
\end{equation}%
where $\Omega \subset 
\mathbb{R}
^{3}$ is the above-mentioned domain.\ 

So, for the proof of the triviality of the solution of the problem (\ref{7.1}%
)-(\ref{7.2}), as usual, is used the method of contradiction. Consequently,
we will start by assuming that this problem has a nontrivial solution,
moreover, will use the feature of having nonlinearity of this problem.

In the beginning, we will study the following quadratic form (\cite{Gan}) 
\begin{equation*}
B\left( w,w\right) =\underset{j,k=1}{\overset{3}{\sum }}\left( \frac{%
\partial v_{k}}{\partial x_{j}}w_{k}w_{j}\right) \left( t,x\right) ,
\end{equation*}%
denoting it as 
\begin{equation*}
B\left( w,w\right) \equiv \underset{j,k=1}{\overset{3}{\sum }}\left(
a_{jk}w_{k}w_{j}\right) \left( t,x\right) .
\end{equation*}%
As is known, the quadratic form $B\left( w,w\right) $ can be transformed
into the canonical form 
\begin{eqnarray*}
B\left( w,w\right) &\equiv &\underset{i=1}{\overset{3}{\sum }}\left(
b_{i}w_{i}^{2}\right) \left( t,x\right) ,\quad b_{i}\left( t,x\right) \equiv
b_{i}\left( \overline{D_{j}v_{k}}\right) , \\
\text{where \ }D_{i}v_{k} &\equiv &\frac{\partial v_{k}}{\partial x_{i}}%
,\quad i,k=1,2,3,\quad b_{i}:%
\mathbb{R}
^{9}\longrightarrow 
\mathbb{R}
\text{ be functions.}
\end{eqnarray*}

Then the matrix $\left\Vert a_{jk}\right\Vert $ of coefficients of the
quadratic form $B\left( w,w\right) $\ can be represented in the following
symmetric form 
\begin{equation*}
\left\Vert a_{jk}\right\Vert _{j,k=1}^{3}=\left\Vert 
\begin{array}{ccc}
a_{11} & a_{12} & a_{13} \\ 
a_{21} & a_{22} & a_{23} \\ 
a_{31} & a_{32} & a_{33}%
\end{array}%
\right\Vert ,\quad \text{where }a_{jk}=a_{kj}=\frac{1}{2}\left(
D_{j}v_{k}+D_{k}v_{j}\right) ,
\end{equation*}%
In this case, the above representation exists according to the symmetry of
matrix $\left\Vert a_{jk}\right\Vert _{j,k=1}^{3}$ (see, \cite{Gan}). Under
such transformation, coefficients $b_{i}$ have the following presentations 
\begin{equation}
b_{1}=D_{1}v_{1};\ b_{2}=D_{2}v_{2}-\frac{\left(
D_{1}v_{2}+D_{2}v_{1}\right) ^{2}}{4b_{1}};\ b_{3}=\frac{\det \left\Vert
D_{i}v_{k}\right\Vert _{i,k=1}^{3}}{\det \left\Vert D_{i}v_{k}\right\Vert
_{i,k=1}^{2}},  \label{7.3a}
\end{equation}%
where $\left\Vert D_{i}v_{k}\right\Vert _{i,k=1}^{3}$\ and $\left\Vert
D_{i}v_{k}\right\Vert _{i,k=1}^{2}$ define by equalities

\begin{center}
$\left\Vert D_{i}v_{k}\right\Vert _{i,k=1}^{3}\equiv \left\Vert 
\begin{array}{ccc}
D_{1}v_{1} & \frac{1}{2}\left( D_{1}v_{2}+D_{2}v_{1}\right) & \frac{1}{2}%
\left( D_{1}v_{3}+D_{3}v_{1}\right) \\ 
\frac{1}{2}\left( D_{1}v_{2}+D_{2}v_{1}\right) & D_{2}v_{2} & \frac{1}{2}%
\left( D_{2}v_{3}+D_{3}v_{2}\right) \\ 
\frac{1}{2}\left( D_{1}v_{3}+D_{3}v_{1}\right) & \frac{1}{2}\left(
D_{2}v_{3}+D_{3}v_{2}\right) & D_{3}v_{3}%
\end{array}%
\right\Vert $
\end{center}

and

\begin{center}
$\left\Vert D_{i}v_{k}\right\Vert _{i,k=1}^{2}\equiv \left\Vert 
\begin{array}{cc}
D_{1}v_{1} & \frac{1}{2}\left( D_{1}v_{2}+D_{2}v_{1}\right) \\ 
\frac{1}{2}\left( D_{1}v_{2}+D_{2}v_{1}\right) & D_{2}v_{2}%
\end{array}%
\right\Vert $ \ 
\end{center}

for any $\left( t,x\right) \in Q^{T}\equiv \left( 0,T\right) \times \Omega $.

Therefore, we have 
\begin{equation}
\left( B\left( w,w\right) \right) \left( t,x\right) \equiv \underset{j,k=1}{%
\overset{3}{\sum }}\left( a_{jk}w_{k}w_{j}\right) \left( t,x\right) \equiv 
\underset{j=1}{\overset{3}{\sum }}b_{j}\left( t,x\right) \cdot
w_{j}^{2}\left( t,x\right)  \label{7.3}
\end{equation}

that can rewrite in the following open form 
\begin{equation*}
B\left( w,w\right) \equiv \frac{1}{D_{1}v_{1}}\left[ 2D_{1}v_{1}w_{1}+\left(
D_{1}v_{2}+D_{2}v_{1}\right) w_{2}+\left( D_{1}v_{3}+D_{3}v_{1}\right) w_{3}%
\right] ^{2}+
\end{equation*}%
\begin{equation*}
\frac{1}{\left( 4D_{1}v_{1}\right) ^{2}}\left( 4D_{1}v_{1}D_{2}v_{2}-\left(
D_{1}v_{2}+D_{2}v_{1}\right) ^{2}\right) \times
\end{equation*}%
\begin{equation*}
\left[ \left( 4D_{1}v_{1}D_{2}v_{2}-\left( D_{1}v_{2}+D_{2}v_{1}\right)
^{2}\right) w_{2}\right. +
\end{equation*}%
\begin{equation*}
\left. \left( 2D_{1}v_{1}\left( D_{2}v_{3}+D_{3}v_{2}\right) -\left(
D_{1}v_{2}+D_{2}v_{1}\right) \left( D_{1}v_{3}+D_{3}v_{1}\right) \right)
w_{3}\right] ^{2}+
\end{equation*}%
\begin{equation*}
\frac{1}{4}\left[ 4D_{1}v_{1}D_{2}v_{2}D_{3}v_{3}+\left(
D_{1}v_{2}+D_{2}v_{1}\right) \left( D_{1}v_{3}+D_{3}v_{1}\right) \left(
D_{2}v_{3}+D_{3}v_{2}\right) \right. -
\end{equation*}%
\begin{equation*}
\left. D_{1}v_{1}\left( D_{2}v_{3}+D_{3}v_{2}\right) ^{2}-D_{2}v_{2}\left(
D_{1}v_{3}+D_{3}v_{1}\right) ^{2}-D_{3}v_{3}\left(
D_{1}v_{2}+D_{2}v_{1}\right) ^{2}\right] w_{3}^{2}.
\end{equation*}

Taking into account (\ref{7.3}) in the equation ((\ref{7.1}) we get 
\begin{equation*}
\frac{1}{2}\frac{\partial }{\partial t}\left\Vert w\right\Vert _{2}^{2}+\nu
\left\Vert \nabla w\right\Vert _{2}^{2}+\underset{j=1}{\overset{3}{\sum }}%
\left\langle b_{j}w_{j},w_{j}\right\rangle =0,\quad \left\Vert
w_{0}\right\Vert _{2}=0,
\end{equation*}%
or 
\begin{equation}
\frac{1}{2}\frac{\partial }{\partial t}\left\Vert w\right\Vert _{2}^{2}=-\nu
\left\Vert \nabla w\right\Vert _{2}^{2}-\underset{j=1}{\overset{3}{\sum }}%
\left\langle b_{j}w_{j},w_{j}\right\rangle ,\quad \left\Vert
w_{0}\right\Vert _{2}=0.  \label{7.4}
\end{equation}

This shows that if $b_{j}\left( t,x\right) \geq 0$ for a.e. $\left(
t,x\right) \in Q^{T}$ then the posed problem has a unique solution. It
should be noted that images of functions $b_{j}\left( t,x\right) $ and $%
D_{i}v_{k}$ belong to the bounded subset of the same space. So, it remains
to investigate the cases when the mentioned inequality isn't fulfilled.

Here, the following variants are possible:

1. Integral of $B\left( w,w\right) $ is determined and non-negative 
\begin{equation*}
\underset{\Omega }{\int }B\left( w,w\right) dx=\underset{j=1}{\overset{3}{%
\sum }}\left\langle b_{j}w_{j},w_{j}\right\rangle \equiv \underset{j=1}{%
\overset{3}{\sum }}{}\underset{\Omega }{\int }b_{j}w_{j}^{2}dx\geq 0;
\end{equation*}

Then the considered problem has a unique weak solution for $t>0$.

2. The above integral is undetermined, and $\underset{j=1}{\overset{3}{\sum }%
}{}\underset{\Omega }{\int }b\ w_{j}^{2}dx\neq 0$.

In this case, for investigation of the problem (\ref{7.4}) it is necessary
to derive suitable estimates for $B\left( w,w\right) \equiv \underset{j,k=1}{%
\overset{3}{\sum }}\left( D_{i}v_{k}w_{k}w_{j}\right) $. So, for the
undetermined part with the trilinear term, using the well-known inequations,
we obtain the following estimation 
\begin{equation*}
\underset{\Omega }{\int }\left\vert B\left( w,w\right) \right\vert dx\leq
c^{2}\underset{i,j=1}{\overset{3}{\sum }}\left\Vert D_{j}v_{i}\right\Vert
_{2}\left\Vert w_{i}\right\Vert _{2}^{\frac{1}{4}}\left\Vert \nabla
w_{i}\right\Vert _{2}^{\frac{3}{4}}\left\Vert w_{j}\right\Vert _{2}^{\frac{1%
}{4}}\left\Vert \nabla w_{j}\right\Vert _{2}^{\frac{3}{4}}.
\end{equation*}

{}Thus, taking into account the above-mentioned estimate, we get the
following inequality 
\begin{equation*}
\frac{1}{2}\frac{\partial }{\partial t}\left\Vert w\left( t\right)
\right\Vert _{2}^{2}\leq -\underset{j=1}{\overset{3}{\sum }}{}\nu \left\Vert
\nabla w_{j}\left( t\right) \right\Vert _{2}^{2}+c^{2}\underset{i,j=1}{%
\overset{3}{\sum }}{}\left\Vert D_{j}v_{i}\left( t\right) \right\Vert
_{2}\left\Vert w_{i}\left( t\right) \right\Vert _{2}^{\frac{1}{2}}\left\Vert
\nabla w_{i}\left( t\right) \right\Vert _{2}^{\frac{3}{2}}
\end{equation*}%
\begin{equation*}
\leq -\underset{j=1}{\overset{3}{\sum }}\left\Vert \nabla w_{j}\left(
t\right) \right\Vert _{2}^{\frac{3}{2}}\left[ \nu \left\Vert \nabla
w_{j}\left( t\right) \right\Vert _{2}^{\frac{1}{2}}-c^{2}\underset{i=1}{%
\overset{3}{\sum }}\left\Vert D_{i}v_{j}\left( t\right) \right\Vert
_{2}\left\Vert w_{j}\left( t\right) \right\Vert _{2}^{\frac{1}{2}}\right]
\end{equation*}%
\begin{equation*}
\leq -\underset{j=1}{\overset{n}{\sum }}\left\Vert \nabla w_{j}\left(
t\right) \right\Vert _{2}^{\frac{3}{2}}\left[ \nu \lambda _{1}^{\frac{1}{4}%
}-c^{2}\underset{i=1}{\overset{n}{\sum }}\left\Vert D_{i}v_{j}\left(
t\right) \right\Vert _{2}\right] \left\Vert w_{j}\left( t\right) \right\Vert
_{2}^{\frac{1}{2}}.
\end{equation*}%
Whence follows, that if $\nu \lambda _{1}^{\frac{1}{4}}\geq c^{2}\underset{%
i=1}{\overset{3}{\sum }}\left\Vert D_{i}v_{j}\left( t\right) \right\Vert
_{2} $ then problem (\ref{2.1})-(\ref{2.3}) has only a unique weak solution
(and this solution is stable), where $\lambda _{1}$ is minimum of the
spectrum (or the first eigenvalue) of the operator Laplace. Thus, Theorem %
\ref{Th_2} is proved.

\subsection{Conclusion II}

From the above consideration follows that the sign of the trilinear term is
dependent on the relations between the accelerations of one speed by the
various directions. This shows that investigation of the process, which is
described by this mathematical model, is necessary for the relations between
the above-mentioned accelerations, and also the corresponding expressions
reduced in (\ref{7.3a}) study, to understand the dynamics of the
investigated process. In other words, this study allows determining the
direction of the move of the studying flow, since this investigation maybe
to show when the solution can be locally unique on some subdomain of $Q^{T}$.

\bigskip

\end{document}